\documentclass[a4paper,11pt]{amsart}

\usepackage{amscd,amssymb}
\usepackage{mathrsfs}
\usepackage{bm}
\usepackage{color}
\usepackage{multirow,bigdelim}
\usepackage[all]{xy}
\usepackage{hyperref}
\usepackage{tikz}

\usepackage{geometry}
\newtheorem{thm}{Theorem}[section]
\newtheorem{lem}[thm]{Lemma}
\newtheorem{cor}[thm]{Corollary}
\newtheorem{prop}[thm]{Proposition}

\newtheorem{conj}[thm]{Conjecture}

\theoremstyle{definition}

\newtheorem{eg}[thm]{Example}
\newtheorem*{conventions}{Conventions}
\newtheorem*{notation}{Notation}
\newtheorem*{acknowledgment}{Acknowledgment}

\theoremstyle{remark}
\newtheorem{rem}[thm]{Remark}

\numberwithin{equation}{section}

\newcommand{\bA}{\mathbb{A}}

\newcommand{\bF}{\mathbb{F}}

\newcommand{\bG}{\mathbb{G}}
\newcommand{\bk}{\Bbbk}

\newcommand{\sO}{\mathscr{O}}
\newcommand{\bP}{\mathbb{P}}

\newcommand{\bQ}{\mathbb{Q}}

\newcommand{\bZ}{\mathbb{Z}}
\newcommand{\Amp}{\mathrm{Amp}}
\newcommand{\Ampc}{\mathrm{Amp}^{\mathrm{cyl}}}

\newcommand{\Cl}{\mathrm{Cl}}

\newcommand{\Sing}{\mathrm{Sing}}

\newcommand{\Spec}{\mathrm{Spec}}
\newcommand{\Supp}{\mathrm{Supp}}

\newcommand{\wC}{\widetilde{C}}
\newcommand{\wD}{\widetilde{D}}
\newcommand{\wE}{\widetilde{E}}
\newcommand{\wF}{\widetilde{F}}

\newcommand{\wS}{\widetilde{S}}
\newcommand{\wGamma}{\widetilde{\Gamma}}
\newcommand{\wDelta}{\widetilde{\Delta}}
\newcommand{\hD}{\hat{D}}
\newcommand{\hE}{\hat{E}}
\newcommand{\hF}{\hat{F}}
\newcommand{\hM}{\hat{M}}
\newcommand{\hR}{\hat{R}}
\newcommand{\hS}{\hat{S}}
\newcommand{\hGamma}{\hat{\Gamma}}

\newcommand{\cF}{\check{F}}
\newcommand{\cM}{\check{M}}
\newcommand{\cGamma}{\check{\Gamma}}

\begin{document}

\title{Cylindrical ample divisors on Du Val del Pezzo surfaces}


\author{}
\address{}
\curraddr{}
\email{}
\thanks{}

\author{Masatomo Sawahara}
\address{Faculty of Education, Hirosaki University, Bunkyocho 1, Hirosaki-shi, Aomori 036-8560, JAPAN}
\curraddr{}
\email{sawahara.masatomo@gmail.com}
\thanks{}



\date{}

\dedicatory{}

\begin{abstract}
Let $S$ be a del Pezzo surface with at worst Du Val singularities of degree $\ge 3$. 
We then construct an $H$-polar cylinder for any ample $\mathbb{Q}$-divisor $H$ on $S$. 
\end{abstract}

\maketitle
\setcounter{tocdepth}{1}

Throughout this article, all considered varieties are assumed to be algebraic and defined over an algebraically closed field $\bk$ of characteristic $0$. 
\section{Introduction}\label{1}
Let $V$ be a normal projective variety. 
We say that an open subset $U$ of $V$ is called a {\it cylinder} if $U$ is isomorphic to $\bA ^1_{\bk} \times Z$ for some variety $Z$.  
Moreover, letting $H$ be an ample $\bQ$-divisor on $V$, we say that a cylinder $U$ of $V$ is an {\it $H$-polar cylinder} if there exists an effective $\bQ$-divisor $D$ on $V$ such that $D \sim _{\bQ} H$ and $V \backslash \Supp (D) = U$. 
We shall consider what kind of polarized cylinders are contained in normal projective varieties. 
The following result motivates our to study polarized cylinders in normal projective varieties: 
\begin{thm}[{\cite[Theorem 2.1]{KPZ14}}\footnote{See also {\cite[3.1.12. Corollary]{KPZ11}}.}]
Let $V$ be a normal projective variety and let $H$ be an ample divisor on $V$. 
Suppose that the generalized affine cone: 
\begin{align*}
\hat{V} := \Spec \left( \bigoplus _{i \ge 0} H^0(V,\sO _V(iH)) \right)
\end{align*}
is normal. 
Then $\hat{V} $admits an effective $\bG _a$-action if and only if $V$ contains an $H$-polar cylinder. 
\end{thm}
Moreover, the geometry of polarized cylinders in projective varieties can be applied to the criterion of flexibility of affine cones (see, e.g., {\cite{Pre13,PW16,Pre21,HHT22,Won22,HT23,HT,Pre}}) and non-trivial $\bG _a$-actions on the complements of hypersurfaces (see, e.g., {\cite{CDP18,Par22}}). 
See also {\cite{CPPZ21}}. 

In order to state our main result and related previous results, we briefly explain cylindrical ample divisors under consideration in this article. 
Let $\Amp (V)$ be the ample cone of a normal rational projective variety $V$. 
In other words, $\Amp (V) := \{ H \in \Cl (V)_{\bQ} \, |\, \text{$H$ is $\bQ$-ample}\}$. 
Then we shall consider the subset $\Ampc (V)$ of $\Amp (V)$ defined by: 
\begin{align*}
\Ampc (V) := \{ H \in \Amp (V) \, |\, \text{there exists an $H$-polar cylinder on $V$}\}.
\end{align*}
In what follows, we say that $\Ampc (V)$ is the {\it set of cylindrical ample divisors} of $V$. 
\begin{rem}
{\cite{CPW17}} and {\cite{MW18}} say that $\Ampc (V)$ is the {\it cone of cylindrical ample divisors} of $V$. 
However, note that $\Ampc (V)$ is not always endowed with a cone structure\footnote{The author is grateful to Professor Karol Palka who pointed out this remark. }. 
Indeed, if $V$ is a smooth cubic surface, then we know $\Ampc (V) = \Amp (V) \backslash \bQ _{>0} [-K_V]$ by Theorem \ref{CPW17} (3) below; in particular, there exist $H_1,H_2 \in \Ampc (V)$ such that $H_1+H_2 \not\in \Ampc (V)$. 
\end{rem}
In previous works, sets of cylindrical ample divisors of smooth rational surfaces are studied by {\cite{KPZ11,KPZ14,CPW16a,CPW17,MW18,Che21}}. 
In particular, the following results about the configurations of sets of cylindrical ample divisors of smooth del Pezzo surfaces hold: 
\begin{thm}\label{CPW17}
Let $S$ be a smooth del Pezzo surface of degree $d$. Then: 
\begin{enumerate}
\item $-K_S \in \Ampc (S)$ if and only if $d \ge 4$. 
\item If $d \ge 4$, then $\Ampc (S) = \Amp (S)$. 
\item If $d = 3$, then $\Ampc (S) = \Amp (S) \backslash \bQ _{>0} [-K_S]$. 
\end{enumerate}
\end{thm}
By considering Theorem \ref{CPW17}, Cheltsov, Park and Won establish the following conjecture: 
\begin{conj}[{\cite[Conjecture 1.2.5]{CPW17}}]\label{conj}
Let $S$ be a log del Pezzo surface (i.e., $S$ is a del Pezzo surface with at most quotient singularities). 
Then $-K_S \in \Ampc (S)$ if and only if $\Ampc (S) = \Amp (S)$. 
\end{conj}
In this article, we study the configuration of sets of cylindrical ample divisors of Du Val del Pezzo surfaces. 
Note that the existing condition of anti-canonical polar cylinders in Du Val del Pezzo surfaces is completely determined as follows: 
\begin{thm}[{\cite[Theorem 1.5]{CPW16b}}]\label{CPW16}
Let $S$ be a Du Val del Pezzo surface of degree $d$ such that $\Sing (S) \not= \emptyset$. Then: 
\begin{enumerate}
\item If $d \ge 3$, then $-K_S \in \Ampc (S)$. 
\item If $d=2$, then $-K_S \not\in \Ampc (S)$ if and only if $S$ allows only a singular point of type $A_1$. 
\item If $d=1$, then $-K_S \not\in \Ampc (S)$ if and only if $S$ allows only a singular point of types $A_1$, $A_2$, $A_3$, $D_4$. 
\end{enumerate}
\end{thm}
In this article, we develop the result of Theorem \ref{CPW16} (1). 
That is, our main theorem is as follows: 
\begin{thm}\label{main(1)}
Let $S$ be a Du Val del Pezzo surface of degree $d \ge 3$ such that $\Sing (S) \not= \emptyset$. 
Then $S$ contains an $H$-polar cylinder for every ample $\bQ$-divisor $H$ on $S$, i.e., $\Ampc (S) = \Amp (S)$. 
\end{thm}
Hence, we have the following corollary: 
\begin{cor}
Conjecture \ref{conj} is true for Du Val del Pezzo surfaces with degree $\ge 3$. 
\end{cor}
\begin{rem}
Let $S$ be a Du Val del Pezzo surface of degree $2$ such that $\Sing (S) \not= \emptyset$ and $-K_S \not\in \Ampc (S)$. 
Then the configuration of $\Ampc (S)$ is partially studied in {\cite{Bel23}}. 
\end{rem}
The article is organized as follows: 
In Section \ref{2}, we review some proprieties of weak del Pezzo surfaces and construct some examples of cylinders in Hirzebruch surfaces. 
In Section \ref{3}, we study a smooth rational surface $\wS$ with a $\bP ^1$-fibration structure. 
In more details, we consider the dimension of some linear systems on $\wS$ and construct some smooth rational curves on $\wS$. 
In Section \ref{4}, we prove Theorem \ref{thm(4)}. 
In other words, letting $S$ be a normal rational surface with several conditions, which almost Du Val del Pezzo surfaces of degree $\ge 3$ satisfy, we will construct an $H$-polar cylinder for any ample $\bQ$-divisor $H$ on $S$. 
In the last Section \ref{5}, we finally prove our main theorem ($=$Theorem \ref{main(1)}) by using results in Sections \ref{3} and \ref{4}. 
\begin{conventions}
A {\it del Pezzo surface} means a normal projective surface such that its anti-canonical divisor is ample. 
A {\it Du Val del Pezzo surface} means a del Pezzo surface with at worst Du Val singularities. 
For a Du Val del Pezzo surface $S$, ${\rm Dyn} (S)$ denotes the Dynkin type of its singularities (for example, ${\rm Dyn}(S)  = A_2+2A_1$ implies that $\Sing (S)$ consists of a Du Val singular point of type $A_2$ and two Du Val singular points of type $A_1$). 
A {\it weak del Pezzo surface} means a smooth projective surface such that its anti-canonical divisor is nef and big. 
For a Du Val del Pezzo surface or a weak del Pezzo surface, the self-intersection number of its (anti-)canonical divisor, which is an integer between $1$ and $9$, is called the {\it degree}. 
For an integer $m$, we say that an $m$-curve on a smooth projective surface is called a smooth rational curve  with self-intersection number $m$. 

For any weighted dual graph, a vertex $\circ$ with the number $m$ corresponds to an $m$-curve. 
Exceptionally, we omit this weight (resp. we omit this weight and use the vertex $\bullet$ instead of $\circ$) if $m=-2$ (resp. $m=-1$).
\end{conventions}
\begin{notation}
We employ the following notations: 
\begin{itemize}
\item $\bA ^d_{\bk}$: the affine space of dimension $d$. 
\item $\bP ^d_{\bk}$: the projective space of dimension $d$. 
\item $\bF _n$: the Hirzebruch surface of degree $n$, i.e., $\bF _n := \bP (\sO _{\bP ^1_{\bk}} \oplus \sO _{\bP ^1_{\bk}}(n))$. 
\item $\bA ^1_{\ast ,\bk}$: the affine line with one closed point removed, i.e., $\bA ^1_{\ast ,\bk} = \Spec (\bk [t^{\pm 1}])$. 
\item $\Cl (X)$: the divisor class group of $X$. 
\item $\Cl (X)_{\bQ} := \Cl (X) \otimes _{\bZ} \bQ$. 
\item $\rho (X)$: the Picard number of $X$. 
\item $K_X$: the canonical divisor on $X$. 
\item $\Sing (X)$: the set of singular points on $X$. 
\item $D_1\sim D_2$: $D_1$ and $D_2$ are linearly equivalent. 
\item $D_1\sim _{\bQ} D_2$: $D_1$ and $D_2$ are $\bQ$-linearly equivalent. 
\item $(D_1 \cdot D_2)$: the intersection number of $D_1$ and $D_2$. 
\item $(D)^2$: the self-intersection number of $D$. 
\item $\varphi ^{\ast}(D)$: the total transform of $D$ by a morphism $\varphi$. 
\item $\psi _{\ast}(D)$: the direct image of $D$ by a morphism $\psi$. 
\item $\Supp (D)$: the support of $D$. 
\item $|D|$: the complete linear system of $D$. 
\item $\sharp D$: the number of all irreducible components in $\Supp (D)$. 
\end{itemize}
\end{notation}
\begin{acknowledgment}
The author would like to thank the referees for useful comments that helped to improve this article. 
The author was supported by the Foundation of Research Fellows, The Mathematical Society of Japan. 
\end{acknowledgment}
\section{Preliminaries}\label{2}
In this subsection, we quickly review the basic but important results. 
Note that the minimal resolution of any Du Val del Pezzo surface is a weak del Pezzo surface. 
Hence, we deal with weak del Pezzo surfaces in this article, so that we summarize some facts about weak del Pezzo surfaces. 
The following result is important to study weak del Pezzo surfaces using explicit birational transformations: 
\begin{lem}\label{lem(2-1)}
Let $V$ be a weak del Pezzo surface. 
Then $V \simeq \bP ^1_{\bk} \times \bP ^1_{\bk}$, $V \simeq \bF _2$ or there exists a birational morphism $h : V \to \bP ^2_{\bk}$. 
\end{lem}
\begin{proof}
See, e.g., {\cite[Theorem 8.1.15]{Dol12}}. 
\end{proof}
Note that the classification of weak del Pezzo surfaces of degree $\ge 3$ is summarized in {\cite{CT88}} and {\cite[\S \S 8.4--8.7 and \S \S 9.2]{Dol12}}. 
Moreover, we will use the following result: 
\begin{lem}\label{lem(2-2)}
Let $V$ be a weak del Pezzo surface and let $\hD$ be a divisor on $V$ such that $(\hD)^2 = (\hD \cdot K_V) = -1$ and $(\hD \cdot \hR) \ge 0$ for every $(-2)$-curve $\hR$ on $V$. 
Then there exists a $(-1)$-curve $\hGamma$ on $V$ such that $\hD \sim \hGamma$. 
\end{lem}
\begin{proof}
See {\cite[Lemma 8.2.22]{Dol12}}. 
\end{proof}
On the other hand, we will construct cylinders in smooth rational surfaces using the following lemma: 
\begin{lem}\label{lem(2-3)}
Let $\hM$ and $\hF$ be the minimal section and a general fiber on the Hirzebruch surface $\bF _n$ of degree $n$, respectively. 
Then the following assertions hold: 
\begin{enumerate}
\item Let $\hF_1,\dots ,\hF _r$ be distinct fibers on $\bF _n$ other than $\hF$. 
Then $\bF _n \backslash (\hM \cup \hF \cup \hF _1 \cup \dots \cup \hF _r) \simeq \bA ^1_{\bk} \times (\bA ^1_{\bk} \backslash \{ r\text{ points}\})$. 
\item Let $\hGamma$ be a smooth rational curve with $\hGamma \sim \hM + n\hF$. 
Then $\bF _n \backslash (\hGamma \cup \hM \cup \hF ) \simeq \bA ^1_{\bk} \times \bA ^1_{\ast ,\bk}$. 
\item Let $\hGamma$ be a smooth rational curve with $\hGamma \sim \hM + (n+1)\hF$. 
Then $\bF _n \backslash (\hGamma \cup \hM \cup \hF _0 ) \simeq \bA ^1_{\bk} \times \bA ^1_{\ast ,\bk}$, where $\hF _0$ is the fiber satisfying $\hM \cap \hGamma \cap \hF _0 \not= \emptyset$. 
\end{enumerate}
\end{lem}
\begin{proof}
In (1) and (2), it seems that these assertions are well-known (see also {\cite{Koj02}}). 

In (3), notice that $\hM \cap \hGamma \cap \hF _0$ consists of exactly one point, say $x$. 
Let $\varphi :V \to \bF _n$ be a blowing-up at $x$ and let $E$ be the exceptional curve of $\varphi$. 
Since $\varphi ^{-1}_{\ast}(\hF _0)$ is a $(-1)$-curve, we obtain a contraction $\psi :V \to \bF _{n+1}$ of $\varphi ^{-1}_{\ast}(\hF _0)$. 
Then $\cM := \psi _{\ast} (\varphi ^{-1}_{\ast}(\hM))$ and $\cF := \psi _{\ast}(E)$ be the minimal section and a fiber of $\bF _{n+1}$, respectively. 
Moreover, $\cGamma := \psi _{\ast} (\varphi ^{-1}_{\ast}(\hGamma))$ is a smooth rational curve on $\bF _{n+1}$ such that  $\cGamma \sim \cM + (n+1)\cF$. 
Hence, by virtue of (2) we have $\bF _n \backslash (\hM \cup \hGamma \cup \hF _0 ) \simeq \bF _{n+1} \backslash (\cGamma \cup \cM \cup \cF ) \simeq \bA ^1_{\bk} \times \bA ^1_{\ast ,\bk}$.  
\end{proof}
\section{Observations of some smooth rational surfaces}\label{3}
Let $\wS$ be a smooth rational surface with a $\bP ^1$-fibration structure $g: \wS \to \bP ^1_{\bk}$. 
For a non-negative integer $n$, we say that $g$ satisfies the condition $(\ast _n)$ if the following three conditions about $g$ hold: 
\begin{itemize}
\item $g$ admits a section with self-intersection number $-n$, say $\wD _0$. 
\item All singular fibers of $g$ consist of only $(-1)$-curves and $(-2)$-curves. 
\item Any irreducible curve on $\wS$ with self-intersection number $\le 2$ is contained in singular fibers of $g$ except for at most one which is $\wD _0$. 
\end{itemize}
In this section, we shall study smooth rational surfaces with a $\bP ^1$-fibration satisfying the condition $(\ast _n)$ for some $n \ge 0$. 
\subsection{Notation on some curves}\label{3-1}
Let $\wS$ be a smooth projective surface with a $\bP ^1$-fibration $g:\wS \to \bP ^1_{\bk}$, which satisfies the condition $(\ast _n)$ for some $n \ge 0$. 
Hence, there is a section $\wD _0$ of $g$ with self-intersection number $-n$. 
Then we notice the following lemma: 
\begin{lem}
Let the notation be the same as above, and let $\wF _i$ be a singular fiber of $g$. 
Then the weighted dual graph of $\wF _i+ \wD _0$ is one of the following: 
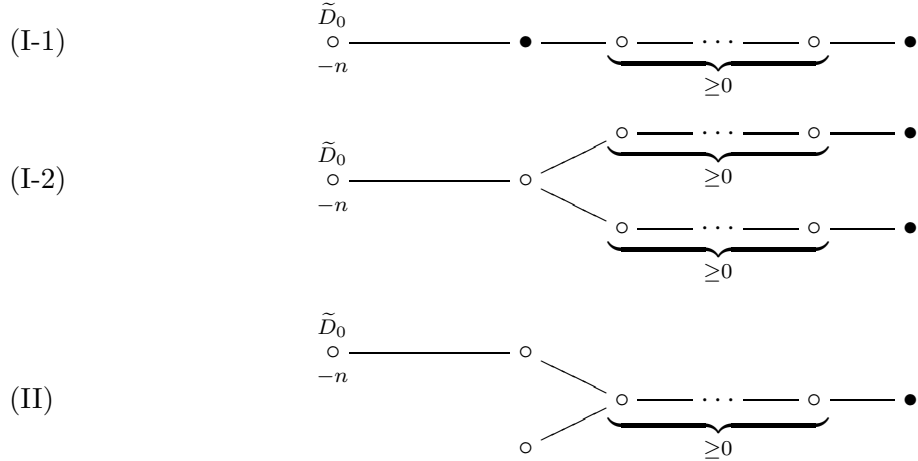
\begin{align}
\label{I-1} &\xygraph{\circ ([]!{+(0,-.3)} {^{-n}}) ([]!{+(0,.25)} {^{\wD_0}}) -[rr] \bullet -[r] \circ -[r] \cdots ([]!{+(0,-.35)} {\underbrace{\ \quad \qquad \qquad \qquad}_{\ge 0}}) -[r] \circ -[r] \bullet} \tag{I-1}\\ 
\label{I-2} &\xygraph{\circ ([]!{+(0,-.3)} {^{-n}}) ([]!{+(0,.25)} {^{\wD_0}}) -[rr] \circ (- []!{+(1,0.5)} \circ -[r] \cdots ([]!{+(0,-.35)} {\underbrace{\ \quad \qquad \qquad \qquad}_{\ge 0}}) -[r] \circ -[r] \bullet, - []!{+(1,-.5)} \circ -[r] \cdots ([]!{+(0,-.35)} {\underbrace{\ \quad \qquad \qquad \qquad}_{\ge 0}}) -[r] \circ -[r] \bullet)} \tag{I-2}\\ 
\label{II} &
\xygraph{\bullet -[l] \circ -[l] \cdots ([]!{+(0,-.4)} {\underbrace{\ \quad \qquad \qquad \qquad}_{\ge 0}}) -[l] \circ (-[]!{+(-1,-0.5)} \circ , -[]!{+(-1,0.5)} \circ -[ll] \circ ([]!{+(0,-.3)} {^{-n}}) ([]!{+(0,.25)} {^{\wD_0}}))} \tag{II}
\end{align}
Here, each vertex with label $\wD _0$ in the above graphs corresponds to the curve $\wD _0$. 
\end{lem}
\begin{proof}
By the condition $(\ast _n)$, every singular fiber of $g$ consists only of $(-1)$-curves and $(-2)$-curves. 
Hence, this lemma follows from {\cite[Lemma 1.5]{Zha88}}. 
\end{proof}
From now on, we prepare some notation. 
First, $\wF$ denotes a general fiber of $g$. 
In what follows, we consider the notation for singular fibers of $g$. 
We fix a section $\wD _0$ of $g$ with self-intersection number $-n$. 
Let $r$, $s$ and $t$ be the numbers of singular fibers of $g$ corresponding to (\ref{I-1}), (\ref{I-2}) and (\ref{II}), respectively, and let $\wF _1,\dots ,\wF _{r+s+t}$ be all singular fibers of $g$. 
Here, if $r>0$ (resp. $s>0$, $t>0$), singular fibers $\{ \wF _i\} _{1 \le i \le r}$ (resp. $\{ \wF _{r+i} \} _{1 \le i \le s}$, $\{ \wF _{r+s+i} \} _{1 \le i \le t}$) correspond to (\ref{I-1}) (resp. (\ref{I-2}), (\ref{II})). 

Suppose $r>0$. 
Then, for $i=1,\dots ,r$ let $\{ \wD _{i,\lambda}\} _{0 \le \lambda \le \alpha _i}$ be all irreducible components of $\wF _i$, where we assume that the weighted dual graph of $\wD _0 + \wF _i$ is as follows: 
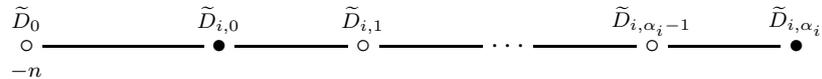
\begin{align*}
\xygraph{\circ ([]!{+(0,-.3)} {^{-n}}) ([]!{+(0,.25)} {^{\wD_0}}) -[rr] \bullet ([]!{+(0,.25)} {^{\wD_{i,0}}}) - []!{+(1.5,0)} \circ ([]!{+(0,.25)} {^{\wD_{i,1}}}) - []!{+(1.5,0)} \cdots - []!{+(1.5,0)} \circ ([]!{+(0,.25)} {^{\wD_{i,\alpha _i-1}}}) - []!{+(1.5,0)} \bullet([]!{+(0,.25)} {^{\wD_{i,\alpha _i}}})}
\end{align*}
Hence, notice that $\sharp \wF _i = \alpha _i + 1$. Furthermore, we put $\wE _i' := \wD _{i,0}$ and $\wE _i := \wD _{i,\alpha _i}$. 

Suppose $s>0$. 
Then, for $i=1,\dots ,s$ let $\{ \wD _{r+i,\lambda}\} _{0 \le \lambda \le \beta _i + \beta _i'}$ be all irreducible components of $\wF _{r+i}$, where we assume that the weighted dual graph of $\wD _0 + \wF _{r+i}$ is as follows: 
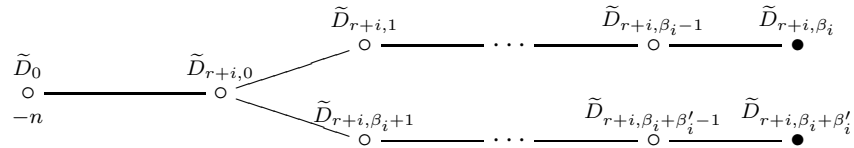
\begin{align*}
\xygraph{\circ ([]!{+(0,-.3)} {^{-n}}) ([]!{+(0,.25)} {^{\wD_0}}) -[rr] \circ ([]!{+(0,.25)} {^{\wD_{r+i,0}}}) (- []!{+(1.5,0.5)} \circ ([]!{+(0,.25)} {^{\wD_{r+i,1}}}) - []!{+(1.5,0)} \cdots - []!{+(1.5,0)} \circ ([]!{+(0,.25)} {^{\wD_{r+i,\beta _i-1}}}) - []!{+(1.5,0)} \bullet ([]!{+(0,.25)} {^{\wD_{r+i,\beta _i}}}), - []!{+(1.5,-.5)} \circ ([]!{+(0,.25)} {^{\wD_{r+i,\beta _i+1}}}) - []!{+(1.5,0)} \cdots - []!{+(1.5,0)} \circ ([]!{+(0,.25)} {^{\wD_{r+i,\beta _i + \beta _i' -1}}}) - []!{+(1.5,0)} \bullet ([]!{+(0,.25)} {^{\wD_{r+i,\beta _i + \beta _i'}}}))}
\end{align*}
and assume further that $\beta _i \ge \beta _i'$. 
Hence, notice that $\sharp \wF _{r+i} = \beta _i + \beta _i' + 1$. Furthermore, we put $\wE _{r+i} := \wD _{r+i,\beta _i}$ and $\wE _{r+i}' := \wD _{r+i,\beta _i + \beta _i'}$. 

Suppose $t>0$. 
Then, for $i=1,\dots ,t$ let $\{ \wD _{r+s+i,\lambda}\} _{0 \le \lambda \le \gamma _i}$ be all irreducible components of $\wF _{r+s+i}$, where we assume that the weighted dual graph of $\wD _0 + \wF _{r+s+i}$ is as follows: 
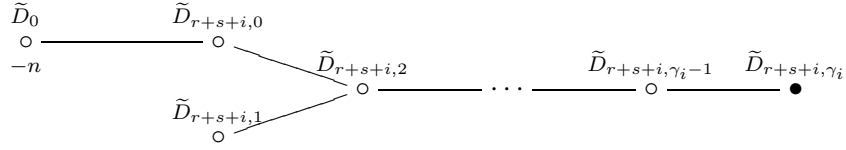
\begin{align*}
\xygraph{\bullet ([]!{+(0,.25)} {^{\wD_{r+s+i,\gamma _i}}}) - []!{+(-1.5,0)} \circ ([]!{+(0,.25)} {^{\wD_{r+s+i,\gamma _i-1}}}) - []!{+(-1.5,0)} \cdots - []!{+(-1.5,0)} \circ ([]!{+(0,.25)} {^{\wD_{r+s+i,2}}}) (- []!{+(-1.5,-0.5)} \circ ([]!{+(0,.25)} {^{\wD_{r+s+i,1}}}),- []!{+(-1.5,0.5)} \circ ([]!{+(0,.25)} {^{\wD_{r+s+i,0}}}) -[ll] \circ ([]!{+(0,-.3)} {^{-n}}) ([]!{+(0,.25)} {^{\wD_0}}))}
\end{align*}
Hence, notice that $\sharp \wF _{r+s+i} = \gamma _i +1$. Furthermore, we put $\wE _{r+s+i} := \wD _{r+s+i,\gamma _i}$. 

Now, set $\alpha := \sum _{i=1}^r\alpha _i$, $\beta := \sum _{i=1}^s\beta _i$, $\beta ':= \sum _{i=1}^s\beta _i'$ and $\gamma := \sum _{i=1}^t\gamma _i$. 
Then we notice $(-K_{\wS})^2 = 8- (\alpha + \beta + \beta ' + \gamma )$. 
\subsection{Constructions of completed linear systems and smooth rational curves}\label{3-2}
In this subsection, we keep the notation from \S \S \ref{3-1}. 
In particular, $\wS$ is a smooth rational surface with a $\bP ^1$-fibration $g:\wS \to \bP ^1_{\bk}$, which satisfies the condition $(\ast _n)$ for some $n \ge 0$. 
We shall prove Lemmas \ref{linear system} and \ref{rational curve}, which will play an important role in the proof of Lemma \ref{lem(4-2)} later. 
\begin{lem}\label{linear system}
With the same notations as above, the following assertions hold: 
\begin{enumerate}
\item Assuming $s=1$, $t=0$ and $\beta ' = 1$, we set the divisor: 
\begin{align*}
\wDelta ^{(1)} := \wD _0 + n\wF - \sum _{i=1}^r \sum _{\lambda = 1}^{\alpha _i}\lambda \wD _{i,\lambda} + \wE _r - \sum _{\mu = 1}^{\beta}\mu \wD _{r+1,\mu}
\end{align*}
on $\wS$. Then $\dim |\wDelta ^{(1)}| \ge n+2 -\alpha -\beta$. 
\item Assuming $s=1$, $t=0$ and $\beta ' = 1$, we set the divisor: 
\begin{align*}
\wDelta ^{(2)} := \wD _0 + n\wF - \sum _{i=1}^r \sum _{\lambda = 1}^{\alpha _i}\lambda \wD _{i,\lambda} - \wE _{r+1}'
\end{align*}
on $\wS$. Then $\dim |\wDelta ^{(2)}| \ge n -\alpha$. 
\item Assuming $s=1$, $t=0$ and $\beta ' = 1$, we set the divisor: 
\begin{align*}
\wDelta ^{(3)} := (\beta -1)\wD _0 + (\beta -1)n\wF - \sum _{i=1}^r \sum _{\lambda = 1}^{\alpha _i}(\beta -1)\lambda \wD _{i,\lambda} - \sum _{\mu =1}^{\beta}(\beta -2)\mu \wD _{r+1,\mu} -\wE _{r+1}'
\end{align*}
on $\wS$. Then $\dim |\wDelta ^{(3)}| \ge \frac{1}{2}\beta(\beta -1)(n+2 -\alpha - \beta) + (\beta -2)$. 
\item Assuming $s=0$, $t=1$ and $\gamma = 2$, we set the divisor:
\begin{align*}
\wDelta ^{(4)} := \wD _0 + n\wF - \sum _{i=1}^r \sum _{\lambda = 1}^{\alpha _i}\lambda \wD _{i,\lambda} + \wE _r - \sum _{\mu =1}^2\wD _{r+1,\mu}
\end{align*}
on $\wS$. Then $\dim |\wDelta ^{(4)}| \ge n + 1 -\alpha$. 
\item Assuming $s=0$, $t=1$ and $\gamma = 3$, we set the divisor:
\begin{align*}
\wDelta ^{(5)} := 2\wD _0 + 2n\wF - \sum _{i=1}^r \sum _{\lambda = 1}^{\alpha _i} 2\lambda \wD _{i,\lambda} + \wE _r - \sum _{\mu = 1}^3\mu \wD _{r+1,\mu}
\end{align*}
on $\wS$. Then $\dim |\wDelta ^{(5)}| \ge 3n +1 -3\alpha$. 
\item Assume $s=t=0$ and $\alpha \ge n+1$. Letting $r' := \min \{ i \in \{1,\dots ,r\} \,|\, \alpha _1 + \dots + \alpha _i \ge n+1\}$ and $\alpha ' := \alpha _1 + \dots + \alpha _{r'}$, we set the divisor:
\begin{align*}
\wDelta ^{(6)} := \wD _0 + n\wF - \sum _{i=1}^{r'} \sum _{\lambda = 1}^{\alpha _i}\lambda \wD _{i,\lambda} + \sum _{\mu = 1}^{\alpha ' - (n+1)}\mu \wD _{r',(n+1) -(\alpha ' - \alpha _{r'}) +\mu}
\end{align*}
on $\wS$, where we consider $\sum _{\mu = 1}^{\alpha ' - (n+1)}\mu \wD _{r',(n+1) -(\alpha ' - \alpha _{r'}) +\mu} = 0$ if $\alpha ' = n+1$. 
Then $\dim |\wDelta  ^{(6)}| \ge 0$. In particular, $|\wDelta ^{(6)}| \not= \emptyset$. 
\end{enumerate}
\end{lem}
\begin{proof}
We only show the assertion (1), the other assertions are similar and left to the reader. 
By the straightforward calculations, we have $(\wDelta ^{(1)} \cdot \wF) = 1$, $(\wDelta ^{(1)})^2 = n - \alpha + 1 - \beta$ and $(\wDelta ^{(1)} \cdot -K_{\wS}) = n+2 - \alpha + 1 - \beta$. 
Since $(\wF \cdot K_{\wS} - \wDelta ^{(1)}) = -3 < 0$, we obtain $h^2(\wS ,\sO _{\wS}(\wDelta ^{(1)})) = 0$ by the Serre duality theorem. 
Notice that $\chi (\wS , \sO _{\wS}) = 1$ because $\wS$ is a smooth rational surface. 
Hence, we have: 
\begin{align*}
\dim |\wDelta ^{(1)}| = h^0(\wS ,\sO _{\wS}(\wDelta ^{(1)})) -1 \ge \chi (\wS ,\sO _{\wS}(\wDelta ^{(1)})) -1 = \frac{1}{2}(\wDelta ^{(1)} \cdot \wDelta ^{(1)} - K_{\wS}) 
= n+2-\alpha -\beta
\end{align*}
by the Riemann-Roch theorem. 
\end{proof}
\begin{rem}
With the same notations as above, we assume further $(-K_{\wS})^2 \ge 5-n$. 
For the reader's convenience, we present figures of general members in linear systems as in Lemma \ref{linear system} according to each assertion: 
\begin{itemize}
\item In the situation of Lemma \ref{linear system} (1)--(2), the configuration of a general member $\wC \in |\wDelta ^{(1)}|$ (resp. $\wC \in |\wDelta ^{(2)}|$) is given as Figure \ref{C} (a) (resp. (b)). 
\item In the situation of Lemma \ref{linear system} (3), for instance when $\beta = 4$, the configuration of a general member $\wC \in |\wDelta ^{(3)}|$ is given as Figure \ref{C} (c). 
\item In the situation of Lemma \ref{linear system} (4)--(5), the configuration of a general member $\wC \in |\wDelta ^{(4)}|$, (resp. $\wC \in |\wDelta ^{(5)}|$) is given as Figure \ref{C} (d) (resp. (e)). 
\item In the situation of Lemma \ref{linear system} (6), for instance when $n=3$, $\alpha = 6$ and $(\alpha _1, \alpha _2 , \alpha _3) = (3,2,1)$, the configuration of a general member $\wC \in |\wDelta ^{(6)}|$ is given as Figure \ref{C} (f).
\end{itemize}
\end{rem}
\begin{figure}[h]
\begin{minipage}[c]{1\hsize}\begin{center}\scalebox{0.7}{\begin{tikzpicture}
\draw [thick] (-1,0)--(7,0);
\node at (-1.4,0) {$\wD _0$};
\node at (-2,6.3) {\Large (a)};
\node at (1.5,3) {\Large $\cdots \cdots$};
\draw [very thick] (-1,5.5)--(3.5,5.5);
\draw [very thick] (4.25,4.5)--(5.25,4.5);
\draw [very thick] (6,5.7)--(7,5.7);
\draw [very thick] (3.5,5.5) .. controls (4,5.5) and (4,4.5) .. (4.25,4.5);
\draw [very thick] (5.25,4.5) .. controls (5.5,4.5) and (5.5,5.7) .. (6,5.7);
\node at (-1.2,5.5) {$\wC$};
\node at (0,-1.5) {};

\draw [dashed] (0.5,-0.25)--(0,1);
\draw (0,0.75)--(0.5,2);
\draw (0.5,1.75)--(0,3);
\draw (0,3.75)--(0.5,5);
\draw [dashed] (0.5,4.75)--(0,6);
\node at (0,3.5) {$\vdots$};
\node at (0.5,-0.6) {\footnotesize $\wE _1'$};
\node at (0,6.3) {\footnotesize $\wE _1$};

\draw [dashed] (3,-0.25)--(2.5,1);
\draw (2.5,0.75)--(3,2);
\draw (3,1.75)--(2.5,3);
\draw (2.5,3.75)--(3,5);
\draw [dashed] (3,4.75)--(2.5,6);
\node at (2.5,3.5) {$\vdots$};
\node at (3,-0.6) {\footnotesize $\wE _{r-1}'$};
\node at (2.5,6.3) {\footnotesize $\wE _{r-1}$};

\draw [dashed] (5,-0.25)--(4.5,1);
\draw (4.5,0.75)--(5,2);
\draw (5,1.75)--(4.5,3);
\draw (4.5,3.75)--(5,5);
\draw [dashed] (5,4.75)--(4.5,6);
\node at (4.5,3.5) {$\vdots$};
\node at (5,-0.6) {\footnotesize $\wE _r'$};
\node at (4,4.2) {\footnotesize $\wD _{r,\alpha _r-1}$};
\node at (4.5,6.3) {\footnotesize $\wE _r$};

\draw (6,-0.25)--(7,1.5);
\draw [dashed] (7,0.25)--(6,1);
\draw (7,1.25)--(6,3);
\draw (6,3.75)--(7,5.5);
\draw [dashed] (7,4.75)--(6,6);
\node at (6,3.4) {$\vdots$};
\node at (6,-0.6) {\footnotesize $\wD_ {r+1,0}$};
\node at (6,1.3) {\footnotesize $\wE _{r+1}'$};
\node at (6,6.3) {\footnotesize $\wE _{r+1}$};
\end{tikzpicture}
\qquad
\begin{tikzpicture}
\draw [thick] (-1,0)--(7,0);
\node at (-1.4,0) {$\wD _0$};
\node at (-2,6.3) {\Large (b)};
\node at (1.5,3) {\Large $\cdots \cdots$};
\draw [very thick] (-1,5.5)--(3.5,5.5);
\draw [very thick] (4.5,4.5)--(4.5,0.5);
\draw [very thick] (3.5,5.5) .. controls (4.5,5.5) .. (4.5,4.5);
\node at (-1.4,5.5) {$\wC$};
\node at (0,-1.5) {};

\draw [dashed] (0.5,-0.25)--(0,1);
\draw (0,0.75)--(0.5,2);
\draw (0.5,1.75)--(0,3);
\draw (0,3.75)--(0.5,5);
\draw [dashed] (0.5,4.75)--(0,6);
\node at (0,3.5) {$\vdots$};
\node at (0.5,-0.6) {\footnotesize $\wE _1'$};
\node at (0,6.3) {\footnotesize $\wE _1$};

\draw [dashed] (3,-0.25)--(2.5,1);
\draw (2.5,0.75)--(3,2);
\draw (3,1.75)--(2.5,3);
\draw (2.5,3.75)--(3,5);
\draw [dashed] (3,4.75)--(2.5,6);
\node at (2.5,3.5) {$\vdots$};
\node at (3,-0.6) {\footnotesize $\wE _r'$};
\node at (2.5,6.3) {\footnotesize $\wE _r$};

\draw (5,-0.25)--(6,1.5);
\draw [dashed] (6,0.25)--(4,1.75);
\draw (6,1.25)--(5,3);
\draw (5,3.75)--(6,5.5);
\draw [dashed] (6,4.75)--(5,6);
\node at (5,3.4) {$\vdots$};
\node at (5,-0.6) {\footnotesize $\wD _{r+1,0}$};
\node at (6.3,0.5) {\footnotesize $\wE _{r+1}'$};
\node at (5,6.3) {\footnotesize $\wE _{r+1}$};
\end{tikzpicture}}\end{center}\end{minipage}

\begin{minipage}[c]{1\hsize}\begin{center}\scalebox{0.7}{\begin{tikzpicture}
\draw [thick] (-1,0)--(7,0);
\node at (-1.4,0) {$\wD _0$};
\node at (-2,6.3) {\Large (c)};
\node at (1.5,3) {\Large $\cdots \cdots$};
\draw [very thick] (-0.5,4.7)--(3.5,4.7);
\draw [very thick] (-0.5,5.1)--(6.25,5.1);
\draw [very thick] (-0.5,5.5)--(6.25,5.5);
\draw [very thick] (-0.5,4.7) .. controls (-0.75,4.7) and (-0.75,5.1) .. (-0.5,5.1);
\draw [very thick] (6.25,5.5) .. controls (6.5,5.5) and (6.5,5.1) .. (6.25,5.1);
\draw [very thick] (4.5,3.7)--(4.5,0.5);
\draw [very thick] (3.5,4.7) .. controls (4.5,4.7) .. (4.5,3.7);
\node at (-1.2,4.9) {$\wC$};
\node at (0,-1.5) {};

\draw [dashed] (0.5,-0.25)--(0,1);
\draw (0,0.75)--(0.5,1.75);
\draw (0.5,1.5)--(0,2.5);
\draw (0,3.25)--(0.5,4.25);
\draw [dashed] (0.5,3.75)--(0,6);
\node at (0,3) {$\vdots$};
\node at (0.5,-0.6) {\footnotesize $\wE _1'$};
\node at (0,6.3) {\footnotesize $\wE _1$};

\draw [dashed] (3,-0.25)--(2.5,1);
\draw (2.5,0.75)--(3,1.75);
\draw (3,1.5)--(2.5,2.5);
\draw (2.5,3.25)--(3,4.25);
\draw [dashed] (3,3.75)--(2.5,6);
\node at (2.5,3) {$\vdots$};
\node at (3,-0.6) {\footnotesize $\wE _r'$};
\node at (2.5,6.3) {\footnotesize $\wE _r$};

\draw (5,-0.25)--(6,1);
\draw [dashed] (6,0.25)--(4,1.75);
\draw (6,0.75)--(5,2);
\draw (5,1.75)--(6,3);
\draw (6,2.75)--(5,4);
\draw [dashed] (5,3.75)--(6,6);
\node at (5,-0.6) {\footnotesize $\wD _{r+1,0}$};
\node at (6.5,0.3) {\footnotesize $\wE _{r+1}'$};
\node at (6,6.3) {\footnotesize $\wE _{r+1}$};
\end{tikzpicture}
\qquad
\begin{tikzpicture}
\draw [thick] (-1,0)--(7,0);
\node at (-1.4,0) {$\wD _0$};
\node at (-2,6.3) {\Large (d)};
\node at (1.5,3) {\Large $\cdots \cdots$};
\draw [very thick] (-1,5.5)--(4,5.5);
\draw [very thick] (5.75,5.5)--(7,5.5);
\draw [very thick] (4.5,4.5)--(5.25,4.5);
\draw [very thick] (4,5.5) .. controls (4.25,5.5) and (4.25,4.5) .. (4.5,4.5);
\draw [very thick] (5.25,4.5) .. controls (5.5,4.5) and (5.5,5.5) .. (5.75,5.5);
\node at (-1.4,5.5) {$\wC$};
\node at (0,-1.5) {};

\draw [dashed] (0.5,-0.25)--(0,1);
\draw (0,0.75)--(0.5,2);
\draw (0.5,1.75)--(0,3);
\draw (0,3.75)--(0.5,5);
\draw [dashed] (0.5,4.75)--(0,6);
\node at (0,3.5) {$\vdots$};
\node at (0.5,-0.6) {\footnotesize $\wE _1'$};
\node at (0,6.3) {\footnotesize $\wE _1$};

\draw [dashed] (3,-0.25)--(2.5,1);
\draw (2.5,0.75)--(3,2);
\draw (3,1.75)--(2.5,3);
\draw (2.5,3.75)--(3,5);
\draw [dashed] (3,4.75)--(2.5,6);
\node at (2.5,3.5) {$\vdots$};
\node at (3,-0.6) {\footnotesize $\wE _{r-1}'$};
\node at (2.5,6.3) {\footnotesize $\wE _{r-1}$};

\draw [dashed] (5,-0.25)--(4.5,1);
\draw (4.5,0.75)--(5,2);
\draw (5,1.75)--(4.5,3);
\draw (4.5,3.75)--(5,5);
\draw [dashed] (5,4.75)--(4.5,6);
\node at (4.5,3.5) {$\vdots$};
\node at (5,-0.6) {\footnotesize $\wE _r'$};
\node at (4,4.2) {\footnotesize $\wD _{r,\alpha _r-1}$};
\node at (4.5,6.3) {\footnotesize $\wE _r$};

\draw (6,-0.25)--(7,2);
\draw [dashed] (6.8,0.9)--(6.8,4.85);
\draw (7,3.75)--(6,6);
\node at (6,-0.6) {\footnotesize $\wD _{r+1,0}$};
\node at (6,6.3) {\footnotesize $\wD _{r+1,1}$};
\node at (6.4,3.3) {\footnotesize $\wE _{r+1}$};
\end{tikzpicture}}\end{center}\end{minipage}

\begin{minipage}[c]{1\hsize}\begin{center}\scalebox{0.7}{\begin{tikzpicture}
\draw [thick] (-1,0)--(7,0);
\node at (-1.4,0) {$\wD _0$};
\node at (-2,6.3) {\Large (e)};
\node at (1.5,3) {\Large $\cdots \cdots$};
\draw [very thick] (-0.5,5.75)--(5.25,5.75);
\draw [very thick] (-0.5,5.25)--(4,5.25);
\draw [very thick] (5.75,5.25)--(5.75,2.5);
\draw [very thick] (5.25,4.25)--(5.25,2.5);
\draw [very thick] (5.25,5.75) .. controls (5.75,5.75) .. (5.75,5.25);
\draw [very thick] (-0.5,5.75) .. controls (-0.75,5.75) and (-0.75,5.25) .. (-0.5,5.25);
\draw [very thick] (5.25,2.5) .. controls (5.25,2) and (5.75,2) .. (5.75,2.5);
\draw [very thick] (5,4.5) .. controls (5.25,4.5) .. (5.25,4.25);
\draw [very thick] (4.5,4.5)--(5,4.5);
\draw [very thick] (4,5.25) .. controls (4.25,5.25) and (4.25,4.5) .. (4.5,4.5);
\node at (-1.2,5.5) {$\wC$};

\draw [dashed] (0.5,-0.25)--(0,1);
\draw (0,0.75)--(0.5,2);
\draw (0.5,1.75)--(0,3);
\draw (0,3.75)--(0.5,5);
\draw [dashed] (0.5,4.75)--(0,6);
\node at (0,3.5) {$\vdots$};
\node at (0.5,-0.6) {\footnotesize $\wE _1'$};
\node at (0,6.3) {\footnotesize $\wE _1$};

\draw [dashed] (3,-0.25)--(2.5,1);
\draw (2.5,0.75)--(3,2);
\draw (3,1.75)--(2.5,3);
\draw (2.5,3.75)--(3,5);
\draw [dashed] (3,4.75)--(2.5,6);
\node at (2.5,3.5) {$\vdots$};
\node at (3,-0.6) {\footnotesize $\wE _{r-1}'$};
\node at (2.5,6.3) {\footnotesize $\wE _{r-1}$};

\draw [dashed] (5,-0.25)--(4.5,1);
\draw (4.5,0.75)--(5,2);
\draw (5,1.75)--(4.5,3);
\draw (4.5,3.75)--(5,5);
\draw [dashed] (5,4.75)--(4.5,6);
\node at (4.5,3.5) {$\vdots$};
\node at (5,-0.6) {\footnotesize $\wE _r'$};
\node at (4,4.1) {\footnotesize $\wD _{r,\alpha _r-1}$};
\node at (4.5,6.3) {\footnotesize $\wE _r$};

\draw (6,-0.25)--(7,2);
\draw (6.7,0.9)--(6.7,4.85);
\draw (7,3.75)--(6,6);
\draw [dashed] (5.6,3)--(7,3);
\node at (6,-0.6) {\footnotesize $\wD _{r+1,0}$};
\node at (6,6.3) {\footnotesize $\wD _{r+1,1}$};
\node at (7,5.1) {\footnotesize $\wD _{r+1,2}$};
\node at (6.3,2.7) {\footnotesize $\wE _{r+1}$};
\end{tikzpicture}
\qquad
\begin{tikzpicture}
\draw [thick] (-1,0)--(7,0);
\node at (-1.4,0) {$\wD _0$};
\node at (-2,6.3) {\Large (f)};
\draw [very thick] (-1,5.5)--(1.5,5.5);
\draw [very thick] (2,3)--(4,3);
\draw [very thick] (4.5,1.5)--(7,1.5);
\draw [very thick] (1.5,5.5) .. controls (1.75,5.5) and (1.75,3) .. (2,3);
\draw [very thick] (4,3) .. controls (4.25,3) and (4.25,1.5) .. (4.5,1.5);
\node at (-1.4,5.5) {$\wC$};

\draw [dashed] (1,-0.25)--(0,1.5);
\draw (0,1.25)--(1,3);
\draw (1,2.75)--(0,4.5);
\draw [dashed] (0,4.25)--(1,6);
\node at (1,-0.6) {\footnotesize $\wE _1'$};
\node at (0.8,1.75) {\footnotesize $\wD _{1,1}$};
\node at (0.8,4) {\footnotesize $\wD _{1,2}$};
\node at (1,6.3) {\footnotesize $\wE _1$};

\draw [dashed] (3.5,-0.25)--(2.5,1.5);
\draw (2.7,0.75)--(2.7,5);
\draw [dashed] (2.5,4.55)--(3.5,6);
\node at (3.5,-0.6) {\footnotesize $\wE _2'$};
\node at (3.2,3.75) {\footnotesize $\wD _{2,1}$};
\node at (3.5,6.3) {\footnotesize $\wE _2$};

\draw [dashed] (6,-0.25)--(5,3.25);
\draw [dashed] (5,2.5)--(6,6);
\node at (6,-0.6) {\footnotesize $\wE _3'$};
\node at (6,6.3) {\footnotesize $\wE _3$};
\end{tikzpicture}}
\caption{Configurations of $\wC$ in Lemma \ref{linear system}}\label{C}
\end{center}\end{minipage}
\end{figure}
\begin{lem}\label{rational curve}
With the same notation as above, the following assertions hold: 
\begin{enumerate}
\item Assume $s=1$, $t=0$, $\beta ' = 1$ and $(-K_{\wS})^2 = 5-n$. Then there exists a $(-1)$-curve $\wGamma$ on $\wS$ such that $\wGamma \sim \wD _0 + n\wF - \sum _{i=1}^r \sum _{\lambda =1}^{\alpha _i}\lambda \wD _{i,\lambda} -\sum _{\mu =1}^{\beta} \mu \wD _{r+1,\mu} + \wE _{r+1}$. 
\item Assume $s=0$, $t=1$, $\gamma \in \{ 2,3\}$ and $(-K_{\wS})^2 = 5-n$. Then there exists a $(-1)$-curve $\wGamma$ on $\wS$ such that $\wGamma \sim \wD _0 + n\wF - \sum _{i=1}^r \sum _{\lambda = 1}^{\alpha _i} \lambda \wD _{i,\lambda} - (\gamma -2)\sum _{\mu =1}^{\gamma}\wD _{r+1,\mu}$. 
\item Assume $s=t=0$ and $(-K_{\wS})^2 \in \{ 6-n,\ 5-n\}$. Then there exists a curve $\wGamma$ on $\wS$ such that $\wGamma \simeq \bP ^1_{\bk}$ and $\wGamma \sim \wD _0 + (n+1)\wF - \sum _{i=1}^r \sum _{\lambda = 1}^{\alpha _i} \lambda \wD _{i,\lambda}$. 
\end{enumerate}
\end{lem}
\begin{proof}
In (1), let $h:\wS \to \hS$ be a of sequence of contractions of $(-1)$-curves and subsequently (smoothly) contractible curves in:
\begin{align*}
\left( \bigsqcup _{i=1}^r\Supp \left( \wE _i' + \sum _{\lambda = 1}^{\alpha _i-1} \wD _{i,\lambda} \right) \right) \sqcup \Supp \left( \wE _{r+1}' + \sum _{\mu = 0}^{\beta -2}\wD _{r+1,\mu} \right) \sqcup \Supp (\wE _{r+1} ),
\end{align*}
where we consider $\sum _{\mu = 0}^{\beta -2}\wD _{r+1,\mu} := 0$ if $\beta =1$. 
By construction of $h$, we know $(h_{\ast}(\wD _0))^2 = -n + \alpha + (\beta -1) = 1$ by virtue of $(-K_{\wS})^2 = 5-n = 8-(\alpha + \beta + 1)$. 
Hence, $\hS \simeq \bF _1$; in particular, there exists a unique $(-1)$-curve $\hGamma$ on $\hS$ further satisfying $(\hGamma \cdot h_{\ast}(\wD _0)) = 0$. 
We put $\wGamma := h^{-1}_{\ast}(\hGamma )$. 
Since $\wGamma$ is a section of $g$, we notice $(\wGamma )^2 \ge -1$ by the last condition of $(\ast _n)$. 
Hence, $\wGamma$ is a $(-1)$-curve because of $-1 = (\hGamma )^2 \ge (\wGamma )^2 \ge -1$. 
By construction of $h$, we easily check $\wGamma \sim \wD _0 + n\wF - \sum _{i=1}^r \sum _{\lambda =1}^{\alpha _i}\lambda \wD _{i,\lambda} -\sum _{\mu =1}^{\beta} \mu \wD _{r+1,\mu} + \wE _{r+1}$. 

In (2), let $h:\wS \to \hS$ be a sequence of contractions of $(-1)$-curves and subsequently (smoothly) contractible curves in:
\begin{align*}
\left( \bigsqcup _{i=1}^r\Supp \left( \wE _i' + \sum _{\lambda = 1}^{\alpha _i-1} \wD _{i,\lambda} \right) \right) \sqcup \Supp \left( \wE _{r+1} + \wD _{r+1,\gamma -1} + \wD _{r+1,\gamma -3} \right), 
\end{align*}
where we consider $\wD _{r+1,\gamma -3} := 0$ if $\gamma = 2$. 
By construction of $h$, we know $(h_{\ast}(\wD _0))^2 = -n + \alpha + (\gamma -2) = 1$ by virtue of $(-K_{\wS})^2 = 5-n = 8-(\alpha + \gamma )$. 
Hence, $\hS \simeq \bF _1$; in particular, there exists a unique $(-1)$-curve $\hGamma$ on $\hS$ further satisfying $(\hGamma \cdot h_{\ast}(\wD _0)) = 0$. 
Then $\wGamma := h^{-1}_{\ast}(\hGamma )$ is a $(-1)$-curve satisfying $\wGamma \sim \wD _0 + n\wF - \sum _{i=1}^r \sum _{\lambda = 1}^{\alpha _i} \lambda \wD _{i,\lambda} - (\gamma -2) \sum _{\mu =1}^{\gamma}\wD _{r+1,\mu}$ by a similar argument to (1). 

In (3), let $h:\wS \to \hS$ be a sequence of contractions of $(-1)$-curves and subsequently (smoothly) contractible curves in:
\begin{align*}
\bigsqcup _{i=1}^r\Supp \left( \wE _i' + \sum _{\lambda = 1}^{\alpha _i-1} \wD _{i,\lambda} \right) .
\end{align*}
By construction of $h$, either $\hS \simeq \bF _0$ or $\hS \simeq \bF _1$ since all irreducible curves on $\hS$ have self-intersection number $\ge -1$. 
Moreover, we know $(h_{\ast}(\wD _0))^2 = -n + \alpha \in \{2,3\}$ by virtue of $(-K_{\wS})^2 = 8-\alpha \in \{ 6-n,5-n\}$. 
Now, we consider two cases $(h_{\ast}(\wD _0))^2 = 3$ and $(h_{\ast}(\wD _0))^2 = 2$ separately. 

Assume $(h_{\ast}(\wD _0))^2 = 3$. 
Then $\hS \simeq \bF _1$. 
Hence, there exists a unique $(-1)$-curve $\hGamma$ on $\hS$ further satisfying $(\hGamma \cdot h_{\ast}(\wD _0)) = 1$. 
Then $\wGamma := h^{-1}_{\ast}(\hGamma )$ is a $(-1)$-curve satisfying $\wGamma \sim \wD _0 + (n+1)\wF - \sum _{i=1}^r \sum _{\lambda = 1}^{\alpha _i} \lambda \wD _{i,\lambda}$ by a similar argument to (1). 

Assume $(h_{\ast}(\wD _0))^2 = 2$. 
Then $\hS \simeq \bF _0$. 
Hence, we take a $0$-curve $\hGamma$ on $\hS$ such that $(\hGamma \cdot h_{\ast}(\wD _0)) = (\hGamma \cdot h_{\ast}(\wF )) = 1$ and $h (\wE _i') \not\in \hGamma$ for $i=1,\dots ,r$. 
Then $\wGamma := h^{-1}_{\ast}(\hGamma )$ is a $0$-curve satisfying $\wGamma \sim \wD _0 + (n+1)\wF - \sum _{i=1}^r \sum _{\lambda = 1}^{\alpha _i} \lambda \wD _{i,\lambda}$. 
\end{proof}
\begin{rem}
For the reader's convenience, we present figures of the smooth rational curve $\wGamma$ according to each assertion of Lemma \ref{rational curve}: 
\begin{itemize}
\item In the situation of Lemma \ref{rational curve} (1), the configuration of $\wGamma \subseteq \wS$ is given as Figure \ref{Gamma} (a). 
\item In the situation of Lemma \ref{rational curve} (2) with $\gamma = 2$ (resp. $\gamma = 3$), the configuration of $\wGamma \subseteq \wS$ is given as Figure \ref{Gamma} (b) (resp. (c)). 
\item In the situation of Lemma \ref{rational curve} (3), the configuration of $\wGamma \subseteq \wS$ is given as Figure \ref{Gamma} (d). 
\end{itemize}
\end{rem}
\begin{figure}[t]
\begin{minipage}[c]{1\hsize}\begin{center}\scalebox{0.7}{\begin{tikzpicture}
\draw [thick] (-1,0)--(7,0);
\node at (-1.4,0) {$\wD _0$};
\node at (-2,6.3) {\Large (a)};
\node at (1.75,3) {\Large $\cdots \cdots$};
\draw [very thick] (-1,5.5)--(4.5,5.5);
\draw [very thick] (5,4.5)--(7,4.5);
\draw [very thick] (4.5,5.5) .. controls (4.75,5.5) and (4.75,4.5) .. (5,4.5);
\node at (-1.4,5.5) {$\wGamma$};
\node at (0,-1.5) {};

\draw [dashed] (0.5,-0.25)--(0,1);
\draw (0,0.75)--(0.5,2);
\draw (0.5,1.75)--(0,3);
\draw (0,3.75)--(0.5,5);
\draw [dashed] (0.5,4.75)--(0,6);
\node at (0,3.5) {$\vdots$};
\node at (0.5,-0.6) {\footnotesize $\wE _1'$};
\node at (0.7,1.25) {\footnotesize $\wD _{1,1}$};
\node at (-0.1,2.25) {\footnotesize $\wD _{1,2}$};
\node at (-0.5,4.3) {\footnotesize $\wD _{1,\alpha _1-1}$};
\node at (0,6.3) {\footnotesize $\wE _1$};

\draw [dashed] (4,-0.25)--(3.5,1);
\draw (3.5,0.75)--(4,2);
\draw (4,1.75)--(3.5,3);
\draw (3.5,3.75)--(4,5);
\draw [dashed] (4,4.75)--(3.5,6);
\node at (3.5,3.5) {$\vdots$};
\node at (4,-0.6) {\footnotesize $\wE _r'$};
\node at (4.2,1.25) {\footnotesize $\wD _{r,1}$};
\node at (3.4,2.25) {\footnotesize $\wD _{r,2}$};
\node at (3,4.3) {\footnotesize $\wD _{r,\alpha _r-1}$};
\node at (3.5,6.3) {\footnotesize $\wE _r$};

\draw (5.5,-0.25)--(6.5,1.5);
\draw [dashed] (6.5,0.25)--(5.5,1);
\draw (6.5,1.25)--(5.5,3);
\draw (5.5,3.75)--(6.5,5.5);
\draw [dashed] (6.5,4.75)--(5.5,6);
\node at (5.5,3.4) {$\vdots$};
\node at (5.5,-0.6) {\footnotesize $\wD_ {r+1,0}$};
\node at (6.7,0.6) {\footnotesize $\wE _{r+1}'$};
\node at (6.5,2.5) {\footnotesize $\wD _{r+1,1}$};
\node at (6.5,4) {\footnotesize $\wD _{r+1,\beta -1}$};
\node at (5.5,6.3) {\footnotesize $\wE _{r+1}$};
\end{tikzpicture}
\qquad
\begin{tikzpicture}
\draw [thick] (-1,0)--(7,0);
\node at (-1.4,0) {$\wD _0$};
\node at (-2,6.3) {\Large (b)};
\node at (1.75,3) {\Large $\cdots \cdots$};
\draw [very thick] (-1,5.5)--(4.5,5.5);
\draw [very thick] (5,0.5)--(7,0.5);
\draw [very thick] (4.5,5.5) .. controls (4.75,5.5) and (4.75,0.5) .. (5,0.5);
\node at (-1.4,5.5) {$\wGamma$};
\node at (0,-1.5) {};

\draw [dashed] (0.5,-0.25)--(0,1);
\draw (0,0.75)--(0.5,2);
\draw (0.5,1.75)--(0,3);
\draw (0,3.75)--(0.5,5);
\draw [dashed] (0.5,4.75)--(0,6);
\node at (0,3.5) {$\vdots$};
\node at (0.5,-0.6) {\footnotesize $\wE _1'$};
\node at (0.7,1.25) {\footnotesize $\wD _{1,1}$};
\node at (-0.1,2.25) {\footnotesize $\wD _{1,2}$};
\node at (-0.5,4.3) {\footnotesize $\wD _{1,\alpha _1-1}$};
\node at (0,6.3) {\footnotesize $\wE _1$};

\draw [dashed] (4,-0.25)--(3.5,1);
\draw (3.5,0.75)--(4,2);
\draw (4,1.75)--(3.5,3);
\draw (3.5,3.75)--(4,5);
\draw [dashed] (4,4.75)--(3.5,6);
\node at (3.5,3.5) {$\vdots$};
\node at (4,-0.6) {\footnotesize $\wE _r'$};
\node at (4.2,1.25) {\footnotesize $\wD _{r,1}$};
\node at (3.4,2.25) {\footnotesize $\wD _{r,2}$};
\node at (3,4.3) {\footnotesize $\wD _{r,\alpha _r-1}$};
\node at (3.5,6.3) {\footnotesize $\wE _r$};

\draw (5.5,-0.25)--(6.5,2);
\draw [dashed] (6.3,0.9)--(6.3,4.85);
\draw (6.5,3.75)--(5.5,6);
\node at (5.5,-0.6) {\footnotesize $\wD _{r+1,0}$};
\node at (5.5,6.3) {\footnotesize $\wD _{r+1,1}$};
\node at (5.8,2.3) {\footnotesize $\wE _{r+1}$};
\end{tikzpicture}}\end{center}\end{minipage}

\begin{minipage}[c]{1\hsize}\begin{center}\scalebox{0.7}{\begin{tikzpicture}
\draw [thick] (-1,0)--(7,0);
\node at (-1.4,0) {$\wD _0$};
\node at (-2,6.3) {\Large (c)};
\node at (1.75,3) {\Large $\cdots \cdots$};
\draw [very thick] (-1,5.5)--(7,5.5);
\node at (-1.4,5.5) {$\wGamma$};

\draw [dashed] (0.5,-0.25)--(0,1);
\draw (0,0.75)--(0.5,2);
\draw (0.5,1.75)--(0,3);
\draw (0,3.75)--(0.5,5);
\draw [dashed] (0.5,4.75)--(0,6);
\node at (0,3.5) {$\vdots$};
\node at (0.5,-0.6) {\footnotesize $\wE _1'$};
\node at (0.7,1.25) {\footnotesize $\wD _{1,1}$};
\node at (-0.1,2.25) {\footnotesize $\wD _{1,2}$};
\node at (-0.5,4.3) {\footnotesize $\wD _{1,\alpha _1-1}$};
\node at (0,6.3) {\footnotesize $\wE _1$};

\draw [dashed] (4,-0.25)--(3.5,1);
\draw (3.5,0.75)--(4,2);
\draw (4,1.75)--(3.5,3);
\draw (3.5,3.75)--(4,5);
\draw [dashed] (4,4.75)--(3.5,6);
\node at (3.5,3.5) {$\vdots$};
\node at (4,-0.6) {\footnotesize $\wE _r'$};
\node at (4.2,1.25) {\footnotesize $\wD _{r,1}$};
\node at (3.4,2.25) {\footnotesize $\wD _{r,2}$};
\node at (3,4.3) {\footnotesize $\wD _{r,\alpha _r-1}$};
\node at (3.5,6.3) {\footnotesize $\wE _r$};

\draw (5.5,-0.25)--(6.5,2);
\draw (6.3,0.9)--(6.3,4.85);
\draw (6.5,3.75)--(5.5,6);
\node at (5.5,-0.6) {\footnotesize $\wD _{r+1,0}$};
\node at (5.5,6.3) {\footnotesize $\wD _{r+1,1}$};
\node at (5.7,2.3) {\footnotesize $\wD _{r+1,2}$};
\draw [dashed] (5,3)--(7,3);
\node at (5,3.3) {\footnotesize $\wE _{r+1}$};
\end{tikzpicture}
\qquad
\begin{tikzpicture}
\draw [thick] (-1,0)--(7,0);
\node at (-1.4,0) {$\wD _0$};
\node at (-2,6.3) {\Large (d)};
\node at (4.25,3) {\Large $\cdots \cdots$};
\draw [very thick] (1.25,5.5)--(7,5.5);
\draw [very thick] (-0.5,-0.25)--(1.25,5.5);
\node at (0.8,5) {$\wGamma$};

\draw [dashed] (3,-0.25)--(2.5,1);
\draw (2.5,0.75)--(3,2);
\draw (3,1.75)--(2.5,3);
\draw (2.5,3.75)--(3,5);
\draw [dashed] (3,4.75)--(2.5,6);
\node at (2.5,3.5) {$\vdots$};
\node at (3,-0.6) {\footnotesize $\wE _1'$};
\node at (3.2,1.25) {\footnotesize $\wD _{1,1}$};
\node at (2.4,2.25) {\footnotesize $\wD _{1,2}$};
\node at (2,4.3) {\footnotesize $\wD _{1,\alpha _1-1}$};
\node at (2.5,6.3) {\footnotesize $\wE _1$};

\draw [dashed] (6.54,-0.25)--(6,1);
\draw (6,0.75)--(6.5,2);
\draw (6.5,1.75)--(6,3);
\draw (6,3.75)--(6.5,5);
\draw [dashed] (6.5,4.75)--(6,6);
\node at (6,3.5) {$\vdots$};
\node at (6.5,-0.6) {\footnotesize $\wE _r'$};
\node at (6.7,1.25) {\footnotesize $\wD _{r,1}$};
\node at (5.9,2.25) {\footnotesize $\wD _{r,2}$};
\node at (5.5,4.3) {\footnotesize $\wD _{r,\alpha _r-1}$};
\node at (6,6.3) {\footnotesize $\wE _r$};
\end{tikzpicture}}
\caption{Configurations of $\wGamma$ in Lemma \ref{rational curve}}\label{Gamma}
\end{center}\end{minipage}
\end{figure}
\subsection{Cases of weak del Pezzo surfaces}\label{3-3}
Let $S$ be a Du Val del Pezzo surface of degree $d \ge 3$ and with $\rho (S)>1$ such that ${\rm Sing}(S) \not= \emptyset$, and let $f:\wS \to S$ be the minimal resolution. 
In this subsection, we shall consider the condition whether there exists a $\bP ^1$-fibration $g: \wS \to \bP ^1_{\bk}$ satisfying $(\ast _n)$ for some $n$. 
Recall that the classification of Du Val del Pezzo surfaces of degree $\ge 3$ (see, e.g., {\cite[\S \S 8.4--8.7 and \S \S 9.2]{Dol12}}). 
In other words, we obtain an explicit birational morphism $h: \wS \to \bP ^2_{\bk}$ according to the pair $(d,{\rm Dyn}(S))$. 
In addition, if $(d,{\rm Dyn}(S)) \not= (3,4A_1)$, there exists a non-negative integer $n$ with $5-d \le n \le 2$ such that we can construct a $\bP ^1$-fibration $g: \wS \to \bP ^1_{\bk}$ admitting a section $\wD _0$ with $(\wD _0)^2 = -n$; moreover, such a $\wD _0$ is unique where $n=2$ (see the Table \ref{table}, for detail construction). 

We shall explain the notation of Table \ref{table}. 
``Type'' means the Dynkin type of $S$. 
``$n$'' means that there exists the number $n$ such that there is a $\bP ^1$-fibration $g:\wS \to \bP ^1_{\bk}$ satisfying $(\ast _n)$. 
In what follows, we fix this $\bP ^1$-fibration $g:\wS \to \bP ^1_{\bk}$ and we use the notation from \S \S \ref{3-1}. 
``$(r,s,t)$'' means the triplet of the numbers of singular fibers of $g$ corresponding to (\ref{I-1}), (\ref{I-2}) and (\ref{II}). 
``Configurations of singular fibers'' means the kind of singular fibers of $g$. 
As an example, we consider the type $A_2+2A_1$. 
Then Table \ref{table} states that there exists a $\bP ^1$-fibration $g:\wS \to \bP ^1_{\bk}$ satisfying $(\ast _2)$ such that $g$ has exactly $(5-d)$ singular fibers as in (\ref{I-1}) and exactly one singular fiber such as in (\ref{II}) but $g$ does not have any singular fiber as in (\ref{I-2}) because of $(r,s,t)=(5-d,0,1)$. 
More strictly, by ``Configurations of singular fibers'', there are exactly $(4-d)$ singular fibers as in (\ref{I-1}) with $\alpha _i = 1$, exactly one singular fiber as in (\ref{I-1}) with $\alpha _i = 2$, and exactly one singular fiber as in (\ref{II}) with $\gamma _1=2$. 
Namely, if $d=3$, the configuration of this $\bP ^1$-fibration $g:\wS \to \bP ^1_{\bk}$ is then as in Figure \ref{A_2+2A_1}, where we assume $\alpha _1 = 1$ and $\alpha _2 = 2$ in this figure. 

By Table \ref{table}, our observation is summarized as follows: 
\begin{prop}\label{prop(3)}
Let $S$ be a Du Val del Pezzo surface of degree $d \ge 3$ and with $\rho (S)>1$ such that ${\rm Sing}(S) \not= \emptyset$, and let $f:\wS \to S$ be the minimal resolution. 
Assume that $(d,{\rm Dyn}(S)) \not= (3,4A_1)$. 
Then there exists a $\bP ^1$-fibration $g: \wS \to \bP ^1_{\bk}$ satisfying $(\ast _n)$ for some non-negative integer $n$ with $5-d \le n \le 2$. 
\end{prop}
\begin{table}[t]
\begin{center}
\begin{tabular}{|c|c|c|c|} \hline

Type & $n$ & $(r,s,t)$ & Configurations of singular fibers \\ \hline \hline
$A_1$ ($d=6$ and 3 lines) & $1$ & $(0,1,0)$ & $(\beta _1, \beta _1') = (1,1)$ \\ \hline
$A_1$ (otherwise) & 2 & $(8-d,0,0)$ & $\{ \alpha _i\} _{1 \le i \le 8-d} = \{ (1)_{8-d}\}$ \\ \hline
$2A_1$ ($d=4$ and 8 lines) & $1$ & $(1,1,0)$ & $\alpha _1 = 2$, $(\beta _1, \beta _1') = (1,1)$ \\ \hline
$2A_1$ (otherwise) & $2$ & $(7-d,0,0)$ & $\{ \alpha _i\} _{1 \le i \le 7-d} = \{ (1)_{6-d},2\}$ \\ \hline
$A_2$ & $2$ & $(6-d,1,0)$ & $\{ \alpha _i\} _{1 \le i \le 6-d} = \{ (1)_{6-d}\}$, $(\beta _1, \beta _1') = (1,1)$ \\ \hline
$3A_1$ & $2$ & $(6-d,0,0)$ & $\{ \alpha _i\} _{1 \le i \le 6-d} = \{ (1)_{4-d}, (2)_2\}$ \\ \hline
$A_2+A_1$ & $2$ & $(6-d,0,0)$ & $\{ \alpha _i\} _{1 \le i \le 6-d} = \{ (1)_{5-d}, 3\}$ \\ \hline
$A_3$ ($d=4$ and 4 lines) & $2$ & $(0,2,0)$ & $\{ (\beta _i , \beta _i')\} _{i=1,2} = \{ (1,1)_2 \}$ \\ \hline
$A_3$ (otherwise) & $2$ & $(5-d,1,0)$ & $\{ \alpha _i\} _{1 \le i \le 5-d} = \{ (1)_{5-d}\}$, $(\beta _1 , \beta _1') = (2,1)$ \\ \hline
$4A_1$ ($d=4$) & $1$ & $(0,0,2)$ & $\{ \gamma _i\} _{i=1,2} = \{ (2)_2\} _{i=1,2}$ \\ \hline
$A_2+2A_1$ & $2$ & $(5-d,0,1)$ & $\{ \alpha _i\} _{1 \le i \le 5-d} = \{ (1)_{4-d},2\}$, $\gamma _1 = 2$\\ \hline
$2A_2$ & $2$ & $(1,1,0)$ & $\alpha _1 = 3$, $(\beta _1,\beta _1') = (1,1)$ \\ \hline
$A_3+A_1$ & $2$ & $(5-d,0,0)$ & $\{ \alpha _i\} _{1 \le i \le 5-d} = \{ (1)_{4-d}, 4\}$ \\ \hline
$A_4$ & $2$ & $(4-d,1,0)$ & $\{ \alpha _i\} _{1 \le i \le 4-d} = \{ (1)_{4-d}\}$, $(\beta _1 , \beta _1') = (3,1)$ \\ \hline
$D_4$ & $2$ & $(4-d,1,0)$ & $\{ \alpha _i\} _{1 \le i \le 4-d} = \{ (1)_{4-d}\}$, $(\beta _1 , \beta _1') = (2,2)$ \\ \hline
$2A_2+A_1$ & $2$ & $(1,0,1)$ & $\alpha _1 =3$, $\gamma _1 = 2$ \\ \hline
$A_3+2A_1$ & $2$ & $(1,0,2)$ & $\alpha _1 =1$, $\{ \gamma _i\} _{i=1,2} = \{ (2)_2\}$ \\ \hline
$A_4+A_1$ & $2$ & $(1,0,1)$ & $\alpha _1 =2$, $\gamma _1 = 3$ \\ \hline
$A_5$ & $2$ & $(0,1,0)$ & $(\beta _1 , \beta _1') = (4,1)$ \\ \hline
$D_5$ & $2$ & $(1,0,1)$ & $\alpha _1 = 1$, $\gamma _1 = 4$ \\ \hline
\end{tabular}
\end{center}
\caption{The list of the configuration of $\bP ^1$-fibration $g:\wS \to \bP ^1_{\bk}$}\label{table}

\end{table}
\begin{figure}[t]
\begin{center}\scalebox{0.5}{\begin{tikzpicture}
\draw [thick] (-0.5,0.5)--(8.5,0.5);
\node at (-0.9,0.5) {\Large $\wD _0$};

\draw [dashed,thick] (1,0)--(0,3);
\draw [dashed,thick] (0,2)--(1,5);
\node at (1,-0.4) {\Large $\wD _{1,0} = \wE _1'$};
\node at (1,5.4) {\Large $\wD _{1,1} = \wE _1$};

\draw [dashed,thick] (4,0)--(3,2);
\draw [thick] (3.2,1.25)--(3.2,3.75);
\draw [dashed,thick] (3,3)--(4,5);
\node at (4,-0.4) {\Large $\wD _{2,0} = \wE _2'$};
\node at (3.8,2.5) {\Large $\wD _{2,1}$};
\node at (4,5.4) {\Large $\wD _{2,2} = \wE _2$};

\draw [thick] (6,0)--(7,2);
\draw [dashed,thick] (6.8,1.25)--(6.8,3.75);
\draw [thick] (7,3)--(6,5);
\node at (6,-0.4) {\Large $\wD_{3,0}$};
\node at (8,2.5) {\Large $\wD _{3,2} = \wE _3$};
\node at (6,5.4) {\Large $\wD _{3,1}$};
\end{tikzpicture}}\end{center}
\caption{Configuration of $\wS$ with $(d,{\rm Dyn} (S)) = (3,A_2+2A_1)$}\label{A_2+2A_1}
\end{figure}
\begin{rem}
Let the notations and the assumptions be the same as in Proposition \ref{prop(3)}, a $\bP ^1$-fibration $g:\wS \to \bP ^1_{\bk}$ satisfying $(\ast _n)$ is not always determined uniquely. 
For instance, we can present the following examples by applying {\cite[Proposition 6.1]{CT88}}: 
\begin{enumerate}
\item If $(d,{\rm Dyn}(S))=(4,3A_1)$, then $\wS$ admits a $\bP ^1$-fibration structure satisfying $(\ast _1)$ and a $\bP ^1$-fibration structure satisfying $(\ast _2)$ (see Figure \ref{3A1}). 
\item If $(d,{\rm Dyn}(S))=(4,A_2+A_1)$, then $\wS$ admits three $\bP ^1$-fibration structures satisfying $(\ast _2)$ (see Figure \ref{A2A1}). 
\end{enumerate}
\end{rem}
\begin{figure}[t]
\begin{minipage}[c]{1\hsize}\begin{center}\scalebox{0.5}{\begin{tikzpicture}
\draw [very thick] (-1,0.5)--(5,0.5);
\node at (-1.4,0.5) {\Large $\wD _0$};

\draw [dashed,thick] (1,0)--(0,2);
\draw [very thick] (0.2,1.25)--(0.2,3.75);
\draw [dashed,thick] (0,3)--(1,5);

\draw [dashed,thick] (4,0)--(3,2);
\draw [very thick] (3.2,1.25)--(3.2,3.75);
\draw [dashed,thick] (3,3)--(4,5);
\end{tikzpicture}
\qquad \qquad
\begin{tikzpicture}
\draw [dashed,very thick] (-1,0.5)--(5,0.5);
\node at (-1.4,0.5) {\Large $\wD _0$};

\draw [dashed] (1,0)--(0,2);
\draw [very thick] (0.2,1.25)--(0.2,3.75);
\draw [dashed,thick] (0,3)--(1,5);

\draw [very thick] (4,0)--(3,2);
\draw [dashed,thick] (3.2,1.25)--(3.2,3.75);
\draw [very thick] (3,3)--(4,5);
\end{tikzpicture}}
\caption{Configuration of $\wS$ with $(d,{\rm Dyn}(S))=(4,3A_1)$}\label{3A1}
\end{center}\end{minipage}

\begin{minipage}[c]{1\hsize}\begin{center}\scalebox{0.5}{\begin{tikzpicture}
\draw [very thick] (-1,0.5)--(5,0.5);
\node at (-1.4,0.5) {\Large $\wD _0$};
\node at (0,6) {};

\draw [dashed,thick] (1,0)--(0,3);
\draw [dashed,thick] (0,2)--(1,5);

\draw [dashed,thick] (4,0)--(3,2);
\draw [very thick] (3,1)--(4,3);
\draw [very thick] (4,2)--(3,4);
\draw [dashed,thick] (3,3)--(4,5);
\end{tikzpicture}
\qquad \qquad
\begin{tikzpicture}
\draw [very thick] (-1,0.5)--(5,0.5);
\node at (-1.4,0.5) {\Large $\wD _0$};
\node at (0,6) {};

\draw [dashed,thick] (1,0)--(0,2);
\draw [very thick] (0.2,1.25)--(0.2,3.75);
\draw [dashed,thick] (0,3)--(1,5);

\draw [very thick] (3.5,0)--(3.5,5);
\draw [dashed,thick] (2.5,2)--(4.5,2);
\draw [dashed,thick] (2.5,3.5)--(4.5,3.5);
\end{tikzpicture}
\qquad \qquad
\begin{tikzpicture}
\draw [very thick] (-1,0.5)--(8,0.5);
\node at (-1.4,0.5) {\Large $\wD _0$};
\node at (0,6) {};

\draw [dashed,thick] (1,0)--(0,3);
\draw [dashed,thick] (0,2)--(1,5);

\draw [dashed,thick] (4,0)--(3,3);
\draw [dashed,thick] (3,2)--(4,5);

\draw [very thick] (7,0)--(6,2);
\draw [dashed,thick] (6.2,1.25)--(6.2,3.75);
\draw [very thick] (6,3)--(7,5);
\end{tikzpicture}}
\caption{Configuration of $\wS$ with $(d,{\rm Dyn}(S))=(4,A_2+A_1)$}\label{A2A1}
\end{center}\end{minipage}
\end{figure}
For the reader's convenience, we present an example of how to construct $\bP ^1$-fibration $g:\wS \to \bP ^1_{\bk}$: 
\begin{eg}
In the case $(d,{\rm Dyn}(S))=(3,A_2+2A_1)$, we shall explicitly construct a $\bP ^1$-fibration $g:\wS \to \bP ^1_{\bk}$ as in Table \ref{table}. 
By Lemma \ref{lem(2-1)}, we obtain a birational morphism $h:\wS \to \bP ^2_{\bk}$. 
Let $e_0$ be a total transform of a general line on $\bP ^2_{\bk}$ by $h$ and let $e_1,\dots ,e_6$ be total transforms of the exceptional divisor. 
By {\cite[\S \S 9.2]{Dol12}}, we can assume that $\wS$ has exactly four $(-2)$-curves, which are linearly equivalent to $e_0 - e_1 -e_2 -e_3$, $e_1-e_2$, $e_2-e_3$ and $e_4-e_5$ by virtue of $|e_0 - e_1 - e_2 - e_3| \not= \emptyset$, $|e_1 - e_2| \not= \emptyset$, $|e_2 - e_3| \not= \emptyset$ and $|e_4 - e_5| \not= \emptyset$. 
Now, we consider a $\bP ^1$-fibration with $(\ast _2)$ admitting a section, which is the $(-2)$-curve corresponding to $e_1 - e_2$. 
By noting $(e_0 - e_1)^2 = 0$, the divisor $e_0 - e_1$ defines a $\bP ^1$-fibration $g := \Phi _{|e_0 - e_1|}:\wS \to \bP ^1_{\bk}$. 
Then we can check that $g$ admits exactly three singular fibers such as Fingure \ref{A_2+2A_1}. 
Moreover, by using Lemma \ref{lem(2-2)} we then know $\wE _1' \sim e_0 -e_1 -e_6$, $\wE _1 \sim e_6$, $\wE _2' \sim e_0 -e_1 -e_4$, $\wD_{2,1} \sim e_4-e_5$, $\wE _2 \sim e_5$, $\wD _{3,0} \sim e_2-e_3$, $\wD _{3,1} \sim e_0 -e_1-e_2 -e_3$ and $\wE _3 \sim e_2$. 
\end{eg}
\section{Cylinders in some singular rational surfaces}\label{4}
The following theorem is our main result in this section: 
\begin{thm}\label{thm(4)}
Let $S$ be a normal rational surface (not necessarily a del Pezzo surface) and let $f: \wS \to S$ be the minimal resolution. 
Assume that the following three conditions hold: 
\begin{itemize}
\item There exists a $\bP ^1$-fibration $g:\wS \to \bP ^1_{\bk}$ satisfying $(\ast _n)$ for some non-negative integer $n$ (see the beginning of \S \ref{3}, for this definition); 
\item All irreducible curves on $\wS$ with self-intersection number $\le -2$ are contracted by $f$; 
\item $(-K_{\wS})^2 \ge 5-n$. 
\end{itemize}
Then $\Ampc (S) = \Amp (S)$. 
\end{thm}
In what follows, we prove Theorem \ref{thm(4)}. 
Let $S$ be a normal rational surface satisfying three conditions as in Theorem \ref{thm(4)} and let $f: \wS \to S$ be the minimal resolution. 
By the assumption, there exists a non-negative integer $n$ such that $\wS$ admits a $\bP ^1$-fibration structure $g:\wS \to \bP ^1_{\bk}$ with $(\ast _n)$. 
Hence, we use the notation from \S \S \ref{3-1} in what follows. 
Put $F := f_{\ast}(\wF)$, $D_0 := f_{\ast}(\wD _0)$, $E_i := f_{\ast}(\wE _i)$ and $E_j' := f_{\ast}(\wE _j')$. 
Then we know $\Cl (S)_{\bQ} = \bQ [D_0] \oplus \bQ [F] \oplus \left( \bigoplus _{i=1}^{r+s}\bQ[E_i] \right)$ by virtue of $E _i' \sim _{\bQ} F - E_i$ for $i=1,\dots ,r+s$ (if $r+s>0$) and $2E_{r+s+j} \sim _{\bQ} F$ for $j=1,\dots ,t$ (if $t>0$). 
Here, we note $D_0 = 0$ when $n \ge 2$. 

We then obtain the following lemma: 
\begin{lem}\label{lem(4-1)}
If one of the following conditions holds: 
\begin{itemize}
\item $\gamma \ge 4$; 
\item $\gamma >0$ and $(-K_{\wS})^2 + \beta ' \ge 6-n$; 
\item $\gamma =0$ and $(-K_{\wS})^2 + \beta ' \ge 7-n$, 
\end{itemize}
then $\Ampc (S) = \Amp (S)$. 
\end{lem}
The above lemma implies that we may treat only the cases $(s,t) = (0,0)$, $(1,0)$ and $(0,1)$. 
Meanwhile, the following lemma holds: 
\begin{lem}\label{lem(4-2)}
If $s+t \le 1$, then $\Ampc (S) = \Amp (S)$. 
\end{lem}
Thus, the proof of Theorem \ref{thm(4)} is completed if Lemmas \ref{lem(4-1)} and \ref{lem(4-2)} are proved. 
Here, Lemma \ref{lem(4-1)} is proved in \S \S \ref{4-1}. 
On the other hand, Lemma \ref{lem(4-2)} is proved in \S \S \ref{4-2}--\S \S \ref{4-5}. 
More precisely, case $n \ge 2$ and $(s,t)=(1,0)$ (resp. $(0,1)$, $(0,0)$) will be treated in \S \S \ref{4-2} (resp. \S \S \ref{4-3}, \S \S \ref{4-4}); moreover, the remaining case $n \le 1$ and $s+t \le 1$ will be treated in \S \S \ref{4-5}. 
\subsection{Case $s+t \ge 2$}\label{4-1}
In this subsection, we prove Lemma \ref{lem(4-1)}. 
We use the notation from \S \S \ref{3-1}. 

Let $H \in \Amp (S)$. 
Since $\Amp (S) \subseteq \Cl (S)_{\bQ} = \bQ [D_0] \oplus \bQ [F] \oplus \left( \bigoplus _{i=1}^r\bQ [E_i] \right) \oplus \left( \bigoplus _{j=1}^s\bQ [E_{r+j}] \right)$, we can write:
\begin{align*}
H \sim _{\bQ} a_0D_0 + aF + \sum _{i=1}^rb_iE_i + \sum _{j=1}^sc_jE_{r+j}
\end{align*}
for some rational numbers $a_0,a,b_1,\dots ,b_r,c_1,\dots ,c_s$, where we consider $a_0 = 0$ if $n \ge 2$. 
Then we note $b_i = -\alpha _i(H \cdot E_i)<0$ for $i=1,\dots ,r$ if $r>0$. 
Put $s' := \sharp \{ j \in \{1,\dots ,s\} \,|\, c_j <0 \}$, where $s' := 0$ if $s=0$. 
Moreover, we can assume $c_j<0$ for $j=1,\dots ,s'$ if $s'>0$. 
\begin{lem}\label{lem(4-1-0)}
We set the divisor: 
\begin{align*}
\wDelta &:= 2\wD _0 + 2n\wF - \sum _{i=1}^r \sum _{\lambda =1}^{\alpha _i}2\lambda \wD _{i,\lambda} \\
&\qquad - \sum _{j=1}^{s'} \sum _{\mu =1}^{\beta _j}2\mu \wD _{r+j,\mu} 
- \sum _{j'=s'+1}^s \sum _{\mu' =1}^{\beta _{j'}'}2\mu' \wD _{r+j',\beta_{j'}+\mu'} 
- \sum _{k=1}^t\sum _{\nu =1}^{\gamma _k}\nu \wD _{r+s+k,\nu}
\end{align*}
on $\wS$. 
Then the following assertions hold: 
\begin{enumerate}
\item $\dim |\wDelta| \ge 3n+2 -3\alpha -3\beta -\gamma$. 
\item If $t=0$, then $\frac{1}{2}\wDelta$ is a $\bZ$-divisor and further $\dim |\frac{1}{2}\wDelta| \ge n+1-\alpha -\beta$. 
\end{enumerate}
\end{lem}
\begin{proof}
Note that: 
\begin{align*}
(\wDelta)^2 &= 4n - \sum _{i=1}^r 4\alpha _i - \sum _{j=1}^{s'} 4\beta _j - \sum _{j'=s'+1}^s 4\beta _{j'}' - \sum _{k=1}^t \gamma _k \ge 4n -4(\alpha + \beta) - \gamma ,\\
(\wDelta \cdot -K_{\wS}) &= 2(n+2) - \sum _{i=1}^r 2\alpha _i - \sum _{j=1}^{s'} 2\beta _j - \sum _{j'=s'+1}^s 2\beta _{j'}' - \sum _{k=1}^t \gamma _k \ge 2(n+2) -2(\alpha + \beta) - \gamma
\end{align*}
because $\beta _{j'} \ge \beta _{j'}'$ for every $j' = s'+1,\dots ,s$ (if $s' \not= s$). 
Hence, assertion (1) follows from the similar argument to the proof of Lemma \ref{linear system}. 
In order to show the assertion (2), we assume $t=0$. 
Then $\frac{1}{2}\wDelta$ is obviously a $\bZ$-divisor on $\wS$ by the definition of $\wDelta$. 
The remaining assertion follows from the similar argument to the proof of Lemma \ref{linear system}. 
\end{proof}
\begin{rem}
For the reader's convenience, we present figures of general members in linear systems related to Lemma \ref{lem(4-1-0)}. 
Let $\wDelta$ be the same as in Lemma \ref{lem(4-1-0)}. 
Then: 
\begin{itemize}
\item The configuration of a general member $\wC \in |\wDelta|$ is given as Figure \ref{C2} (a). 
\item Assuming $t=0$, the configuration of a general member $\wC \in |\frac{1}{2}\wDelta|$ is given as Figure \ref{C2} (b). 
\end{itemize}
\end{rem}
\begin{figure}[h]
\begin{minipage}[c]{1\hsize}\begin{center}\scalebox{0.7}{\begin{tikzpicture}
\draw [thick] (-1,0)--(19,0);
\node at (-1.4,0) {$\wD _0$};
\node at (-2,6.3) {\Large (a)};
\node at (1.5,3.4) {\large $\cdots \cdots$};
\node at (5,3) {\Large $\cdots$};
\node at (10,2.7) {\Large $\cdots$};
\node at (15,3) {\large $\cdots \cdots$};
\draw [very thick] (-0.25,5.25)--(19,5.25);
\draw [very thick] (-0.25,5.75)--(13,5.75);
\draw [very thick] (-0.25,5.25) .. controls (-0.5,5.25) and (-0.5,5.75) .. (-0.25,5.75);
\node at (-0.8,5.5) {$\wC$};
\node at (0,-1.5) {};

\draw [dashed] (0,-0.25)--(0.5,1);
\draw (0.5,0.75)--(0,2);
\draw (0,1.75)--(0.5,3);
\draw (0.5,3.75)--(0,5);
\draw [dashed] (0,4.75)--(0.5,6);
\node at (0.5,3.45) {$\vdots$};
\node at (0,-0.6) {\footnotesize $\wE _1'$};
\node at (0.5,6.3) {\footnotesize $\wE _1$};

\draw [dashed] (2,-0.25)--(2.5,1);
\draw (2.5,0.75)--(2,2);
\draw (2,1.75)--(2.5,3);
\draw (2.5,3.75)--(2,5);
\draw [dashed] (2,4.75)--(2.5,6);
\node at (2.5,3.5) {$\vdots$};
\node at (2,-0.6) {\footnotesize $\wE _r'$};
\node at (2.5,6.3) {\footnotesize $\wE _r$};

\draw (3.5,-0.25)--(3.5,1);
\draw (3.5,1) .. controls (3.5,1.25) .. (3.75,1.25);
\draw (3.75,1.25)--(4.5,1.25);
\draw (3.75,1)--(4,1.75);
\draw (4.25,1)--(4.5,1.75);
\draw (3.75,2.25)--(4,1.5);
\draw (4.25,2.25)--(4.5,1.5);
\node at (3.75,2.75) {$\vdots$};
\node at (3.75,3.25) {$\vdots$};
\node at (4.25,2.75) {$\vdots$};
\node at (4.25,3.25) {$\vdots$};
\draw (3.75,3.5)--(4,4.25);
\draw (4.25,3.5)--(4.5,4.25);
\draw (3.75,4.75)--(4,4);
\draw [dashed] (4.25,4.75)--(4.5,4);
\draw [dashed] (3.75,4.25)--(4.25,6);
\node at (3.5,-0.6) {\footnotesize $\wD _{r+1,0}$};
\node at (4.9,4.5) {\footnotesize $\wE_{r+1}'$};
\node at (4.5,6.3) {\footnotesize $\wE _{r+1}$};

\draw (5.5,-0.25)--(5.5,1);
\draw (5.5,1) .. controls (5.5,1.25) .. (5.75,1.25);
\draw (5.75,1.25)--(6.5,1.25);
\draw (5.75,1)--(6,1.75);
\draw (6.25,1)--(6.5,1.75);
\draw (5.75,2.25)--(6,1.5);
\draw (6.25,2.25)--(6.5,1.5);
\node at (5.75,2.75) {$\vdots$};
\node at (5.75,3.25) {$\vdots$};
\node at (6.25,2.75) {$\vdots$};
\node at (6.25,3.25) {$\vdots$};
\draw (5.75,3.5)--(6,4.25);
\draw (6.25,3.5)--(6.5,4.25);
\draw (5.75,4.75)--(6,4);
\draw [dashed] (6.25,4.75)--(6.5,4);
\draw [dashed] (5.75,4.25)--(6.25,6);
\node at (5.5,-0.6) {\footnotesize $\wD _{r+s',0}$};
\node at (6.9,4.5) {\footnotesize $\wE_{r+s'}'$};
\node at (6.5,6.3) {\footnotesize $\wE _{r+s'}$};

\draw (8.25,-0.25)--(8.25,1);
\draw (8.25,1) .. controls (8.25,1.25) .. (8.5,1.25);
\draw (8.5,1.25)--(9.25,1.25);
\draw (8.5,1)--(8.75,1.75);
\draw (9,1)--(9.25,1.75);
\draw (8.5,2.25)--(8.75,1.5);
\draw (9,2.25)--(9.25,1.5);
\node at (8.5,2.75) {$\vdots$};
\node at (9,2.75) {$\vdots$};
\draw (8.5,3)--(8.75,3.75);
\draw (9,3)--(9.5,4.75);
\draw (8.5,4.25)--(8.75,3.5);
\draw [dashed] (8.75,4.75)--(8.5,4);
\draw [dashed] (9,6)--(9.5,4.25);
\node at (8.25,-0.6) {\footnotesize $\wD _{r+s'+1,0}$};
\node at (9.2,6.3) {\footnotesize $\wE_{r+s'+1}'$};
\node at (8,4.75) {\footnotesize $\wE _{r+s'+1}$};

\draw (10.75,-0.25)--(10.75,1);
\draw (10.75,1) .. controls (10.75,1.25) .. (11,1.25);
\draw (11,1.25)--(11.75,1.25);
\draw (11,1)--(11.25,1.75);
\draw (11.5,1)--(11.75,1.75);
\draw (11,2.25)--(11.25,1.5);
\draw (11.5,2.25)--(11.75,1.5);
\node at (11,2.75) {$\vdots$};
\node at (11.5,2.75) {$\vdots$};
\draw (11,3)--(11.25,3.75);
\draw (11.5,3)--(12,4.75);
\draw (11,4.25)--(11.25,3.5);
\draw [dashed] (11.25,4.75)--(11,4);
\draw [dashed] (11.5,6)--(12,4.25);
\node at (10.75,-0.6) {\footnotesize $\wD _{r+s,0}$};
\node at (11.5,6.3) {\footnotesize $\wE_{r+s}'$};
\node at (10.7,4.75) {\footnotesize $\wE _{r+s}$};

\draw (13,-0.25)--(13,1.5);
\draw (12.75,1.25)--(14,1.25);
\draw (13.75,1.5)--(13.75,0.25);
\draw (13.25,1)--(13.5,1.75);
\draw (13.25,2.25)--(13.5,1.5);
\node at (13.25,2.75) {$\vdots$};
\node at (13.25,3.25) {$\vdots$};
\draw (13.25,3.5)--(13.5,4.25);
\draw (13.25,4.75)--(13.5,4);
\draw [dashed] (13.25,4.25)--(13.75,6);
\node at (13,-0.6) {\footnotesize $\wD _{r+s+1,0}$};
\node at (14.7,1.25) {\footnotesize $\wD _{r+s+1,2}$};
\node at (14.6,0.5) {\footnotesize $\wD _{r+s+1,1}$};
\node at (13.75,6.3) {\footnotesize $\wE _{r+s+1}$};

\draw (16,-0.25)--(16,1.5);
\draw (15.75,1.25)--(17,1.25);
\draw (16.75,1.5)--(16.75,0.25);
\draw (16.25,1)--(16.5,1.75);
\draw (16.25,2.25)--(16.5,1.5);
\node at (16.25,2.75) {$\vdots$};
\node at (16.25,3.25) {$\vdots$};
\draw (16.25,3.5)--(16.5,4.25);
\draw (16.25,4.75)--(16.5,4);
\draw [dashed] (16.25,4.25)--(16.75,6);
\node at (16,-0.6) {\footnotesize $\wD _{r+s+t,0}$};
\node at (17.7,1.25) {\footnotesize $\wD _{r+s+t,2}$};
\node at (17.6,0.5) {\footnotesize $\wD _{r+s+t,1}$};
\node at (16.75,6.3) {\footnotesize $\wE _{r+s+t}$};
\end{tikzpicture}}\end{center}\end{minipage}

\begin{minipage}[c]{1\hsize}\begin{center}\scalebox{0.7}{\begin{tikzpicture}
\draw [thick] (-1,0)--(13,0);
\node at (-1.4,0) {$\wD _0$};
\node at (-2,6.3) {\Large (b)};
\node at (1.5,3.4) {\large $\cdots \cdots$};
\node at (5,3) {\Large $\cdots$};
\node at (10,2.7) {\Large $\cdots$};
\draw [very thick] (-0.25,5.55)--(13,5.5);
\node at (-0.6,5.5) {$\wC$};

\draw [dashed] (0,-0.25)--(0.5,1);
\draw (0.5,0.75)--(0,2);
\draw (0,1.75)--(0.5,3);
\draw (0.5,3.75)--(0,5);
\draw [dashed] (0,4.75)--(0.5,6);
\node at (0.5,3.45) {$\vdots$};
\node at (0,-0.6) {\footnotesize $\wE _1'$};
\node at (0.5,6.3) {\footnotesize $\wE _1$};

\draw [dashed] (2,-0.25)--(2.5,1);
\draw (2.5,0.75)--(2,2);
\draw (2,1.75)--(2.5,3);
\draw (2.5,3.75)--(2,5);
\draw [dashed] (2,4.75)--(2.5,6);
\node at (2.5,3.5) {$\vdots$};
\node at (2,-0.6) {\footnotesize $\wE _r'$};
\node at (2.5,6.3) {\footnotesize $\wE _r$};

\draw (3.5,-0.25)--(3.5,1);
\draw (3.5,1) .. controls (3.5,1.25) .. (3.75,1.25);
\draw (3.75,1.25)--(4.5,1.25);
\draw (3.75,1)--(4,1.75);
\draw (4.25,1)--(4.5,1.75);
\draw (3.75,2.25)--(4,1.5);
\draw (4.25,2.25)--(4.5,1.5);
\node at (3.75,2.75) {$\vdots$};
\node at (3.75,3.25) {$\vdots$};
\node at (4.25,2.75) {$\vdots$};
\node at (4.25,3.25) {$\vdots$};
\draw (3.75,3.5)--(4,4.25);
\draw (4.25,3.5)--(4.5,4.25);
\draw (3.75,4.75)--(4,4);
\draw [dashed] (4.25,4.75)--(4.5,4);
\draw [dashed] (3.75,4.25)--(4.25,6);
\node at (3.5,-0.6) {\footnotesize $\wD _{r+1,0}$};
\node at (4.9,4.5) {\footnotesize $\wE_{r+1}'$};
\node at (4.5,6.3) {\footnotesize $\wE _{r+1}$};

\draw (5.5,-0.25)--(5.5,1);
\draw (5.5,1) .. controls (5.5,1.25) .. (5.75,1.25);
\draw (5.75,1.25)--(6.5,1.25);
\draw (5.75,1)--(6,1.75);
\draw (6.25,1)--(6.5,1.75);
\draw (5.75,2.25)--(6,1.5);
\draw (6.25,2.25)--(6.5,1.5);
\node at (5.75,2.75) {$\vdots$};
\node at (5.75,3.25) {$\vdots$};
\node at (6.25,2.75) {$\vdots$};
\node at (6.25,3.25) {$\vdots$};
\draw (5.75,3.5)--(6,4.25);
\draw (6.25,3.5)--(6.5,4.25);
\draw (5.75,4.75)--(6,4);
\draw [dashed] (6.25,4.75)--(6.5,4);
\draw [dashed] (5.75,4.25)--(6.25,6);
\node at (5.5,-0.6) {\footnotesize $\wD _{r+s',0}$};
\node at (6.9,4.5) {\footnotesize $\wE_{r+s'}'$};
\node at (6.5,6.3) {\footnotesize $\wE _{r+s'}$};

\draw (8.25,-0.25)--(8.25,1);
\draw (8.25,1) .. controls (8.25,1.25) .. (8.5,1.25);
\draw (8.5,1.25)--(9.25,1.25);
\draw (8.5,1)--(8.75,1.75);
\draw (9,1)--(9.25,1.75);
\draw (8.5,2.25)--(8.75,1.5);
\draw (9,2.25)--(9.25,1.5);
\node at (8.5,2.75) {$\vdots$};
\node at (9,2.75) {$\vdots$};
\draw (8.5,3)--(8.75,3.75);
\draw (9,3)--(9.5,4.75);
\draw (8.5,4.25)--(8.75,3.5);
\draw [dashed] (8.75,4.75)--(8.5,4);
\draw [dashed] (9,6)--(9.5,4.25);
\node at (8.25,-0.6) {\footnotesize $\wD _{r+s'+1,0}$};
\node at (9.2,6.3) {\footnotesize $\wE_{r+s'+1}'$};
\node at (8,4.75) {\footnotesize $\wE _{r+s'+1}$};

\draw (10.75,-0.25)--(10.75,1);
\draw (10.75,1) .. controls (10.75,1.25) .. (11,1.25);
\draw (11,1.25)--(11.75,1.25);
\draw (11,1)--(11.25,1.75);
\draw (11.5,1)--(11.75,1.75);
\draw (11,2.25)--(11.25,1.5);
\draw (11.5,2.25)--(11.75,1.5);
\node at (11,2.75) {$\vdots$};
\node at (11.5,2.75) {$\vdots$};
\draw (11,3)--(11.25,3.75);
\draw (11.5,3)--(12,4.75);
\draw (11,4.25)--(11.25,3.5);
\draw [dashed] (11.25,4.75)--(11,4);
\draw [dashed] (11.5,6)--(12,4.25);
\node at (10.75,-0.6) {\footnotesize $\wD _{r+s,0}$};
\node at (11.5,6.3) {\footnotesize $\wE_{r+s}'$};
\node at (10.7,4.75) {\footnotesize $\wE _{r+s}$};
\end{tikzpicture}}
\caption{Configurations of $\wC$ in Lemma \ref{lem(4-1-0)}}\label{C2}
\end{center}\end{minipage}
\end{figure}
For simplicity, we set $d_{r,s'} := a + \sum _{i=1}^rb_i + \sum _{j=1}^{s'}c_j $, where we consider $\sum _{j=1}^{s'}c_j = 0$ if $s'=0$. 
\begin{lem}\label{lem(4-1-1)}
If one of the following conditions holds: 
\begin{enumerate}
\item $\gamma \ge 4$; 
\item $\gamma > 0$ and $(-K_{\wS})^2 + \beta ' \ge 6-n$; 
\item $\gamma =0$ and $(-K_{\wS})^2 + \beta ' \ge 7-n$, 
\end{enumerate}
then $d_{r,s'} >0$. 
\end{lem}
\begin{proof}
Let $\wDelta$ be the same as in Lemma \ref{lem(4-1-0)}. 
In this proof, we will use the formula $(-K_{\wS})^2 = 8 - (\alpha + \beta + \beta ' + \gamma ) \ge 5-n$. 

In (1) and (2), we then obtain $\dim |\wDelta| \ge 0$. 
Indeed, if $\gamma \ge 4$, then we have $\dim |\wDelta| \ge 2\gamma -7 \ge 1$ by Lemma \ref{lem(4-1-0)} (1) combined with $-(\alpha + \beta) \ge -3-n + \gamma$. 
Meanwhile, if $\gamma > 0$ and $(-K_{\wS})^2 + \beta' \ge 6-n$, then we have $\dim |\wDelta | \ge 2(\gamma -2) \ge 0$ by Lemma \ref{lem(4-1-0)} (1) combined with $-(\alpha + \beta) \ge -2-n + \gamma$. 
Here, we note that $\gamma \ge 2$ provided $\gamma > 0$. 
Hence, $|\wDelta| \not= \emptyset$, so that we can take a general member $\wC$ of $|\wDelta|$. 
Notice $f_{\ast}(\wC ) \sim _{\bQ} 2nF - \sum _{i=1}^r 2\alpha _i E_i - \sum _{j=1}^{s'} 2\beta _j E_{r+j} - \sum _{j'=s'+1}^s 2\beta _{j'}' E_{r+j'}' - \sum _{k=1}^t \gamma _k E_{r+s+k} \not\sim _{\bQ} 0$. 
Thus, we obtain $0 < \frac{1}{2}(H \cdot f_{\ast}(\wC )) = d_{r,s'}$ (see also Figure \ref{C2} (a)). 

In (3), we then obtain $\dim | \frac{1}{2}\wDelta| \ge 0$ by Lemma \ref{lem(4-1-0)} (2) combined with $-(\alpha + \beta) \ge -1-n$. 
Hence, $|\frac{1}{2}\wDelta| \not= \emptyset$, so that we can take a general member $\wC$ of $|\frac{1}{2}\wDelta|$. 
Notice $f_{\ast}(\wC ) \sim _{\bQ} nF - \sum _{i=1}^r \alpha _i E_i - \sum _{j=1}^{s'} \beta _j E_{r+j} - \sum _{j'=s'+1}^s \beta _{j'}' E_{r+j'}' \not\sim _{\bQ} 0$. 
Thus, we obtain $0 < (H \cdot f_{\ast}(\wC )) = d_{r,s'}$ (see also Figure \ref{C2} (b)). 
\end{proof}
Now, we shall consider $U := f(\wS \backslash \Supp (\wD _0 + \wF + \wF _1 +\dots + \wF _{r+s+t}))$. 
Note that $U$ is a cylinder in $S$ because $U \simeq \bA ^1_{\bk} \times (\bA ^1_{\bk} \backslash \{ (r+s+t)\text{ points}\})$ by Lemma \ref{lem(2-3)} (1). 
Then the following result holds: 
\begin{lem}\label{lem(4-1-2)}
If $d_{r,s'} >0$, then $U$ is an $H$-polar cylinder; in particular, $H \in \Ampc (S)$. 
\end{lem}
\begin{proof}
Note that $a_0 = (H \cdot F) >0$ when $n \le 1$. 
We take the effective $\bQ$-divisor:
\begin{align*}
D := &a_0D_0 + \sum _{i=1}^r (-b_i) E_i' + \sum _{j=1}^{s'} (-c_j) E_{r+j}' + \sum _{j'=s'+1}^s c_{j'} E_{r+j'} \\
&\quad +\frac{d_{r,s'}}{r+s+t+1} \left \{F + \sum _{k=1}^{r+s}(E_k+E_k') + \sum _{\ell =1}^t2E_{r+s+\ell} \right\}
\end{align*}
on $S$. 
Then $D \sim _{\bQ} H$. 
Moreover, we know $S \backslash \Supp (D) = U$. 
Thus, $U$ is an $H$-polar cylinder. 
\end{proof}
Lemma \ref{lem(4-1)} follows from Lemmas \ref{lem(4-1-1)} and \ref{lem(4-1-2)}. 
\begin{rem}
Lemma \ref{lem(4-1-2)} states that $U$ is an $H$-polar cylinder for every ample $\bQ$-divisor $H$ on $S$ when $d_{r,s'} > 0$. 
\end{rem}
\subsection{Case $(s,t)=(1,0)$ and $n \ge 2$}\label{4-2}
In this subsection, we prove the case of $(s,t)=(1,0)$ in Lemma \ref{lem(4-2)} under the assumption that $n \ge 2$. 
We use the notation from \S \S \ref{3-1}. 
Assume that $(s,t)=(1,0)$ and $n \ge 2$. 
If $(-K_{\wS})^2 \ge 6-n$ or $\beta ' \ge 2$, then we obtain $\Ampc (S) = \Amp (S)$ by Lemma \ref{lem(4-1)} because $(-K_{\wS})^2 + \beta ' \ge 7-n$. 
From now on, we assume further that $(-K_{\wS})^2 = 5-n$ and $\beta ' = 1$; namely, $\alpha + \beta = n+2$. 

Let $H \in \Amp (S)$. 
Since $\Amp (S) \subseteq \Cl (S)_{\bQ} = \bQ [F] \oplus \left( \bigoplus _{i=1}^r\bQ [E_i] \right) \oplus \bQ [E_{r+1}]$, we can write:
\begin{align*}
H \sim _{\bQ} aF + \sum _{i=1}^rb_iE_i + cE_{r+1}
\end{align*}
for some rational numbers $a,b_1,\dots ,b_r,c$. 
We note $b_i = -\alpha _i(H \cdot E_i)<0$ for $i=1,\dots ,r$ if $r>0$. 
Then the following lemma holds: 
\begin{lem}\label{lem(4-2-1)}
The following assertions hold: 
\begin{enumerate}
\item $a + \sum _{i=1}^rb_i - \frac{b_r}{\alpha _r} + c > 0$ and $a + \sum _{i=1}^rb_i - \frac{b_r}{\alpha _r} > 0$. 
\item If $\beta \ge 2$, then $a + \sum _{i=1}^rb_i > 0$. 
\item If $\beta \ge 2$, then $a + \sum _{i=1}^rb_i + \frac{\beta -2}{\beta -1}c > 0$. 
\end{enumerate}
\end{lem}
\begin{proof}
Let $\widetilde{\Delta}^{(1)}$, $\widetilde{\Delta}^{(2)}$ and $\widetilde{\Delta}^{(3)}$ be the same as in Lemma \ref{linear system} (1), (2) and (3), respectively. 

Then we obtain $\dim |\wDelta ^{(1)}| \ge 0$ by Lemma \ref{linear system} (1) combined with $n+2=\alpha + \beta$. 
Hence, $|\wDelta ^{(1)}| \not= \emptyset$, so that we can take a general member $\wC$ of $|\widetilde{\Delta}^{(1)}|$. 
Notice $f_{\ast}(\wC ) \sim _{\bQ} nF -\sum _{i=1}^r \alpha _i E_i +E_r - \beta E_{r+1} \not\sim _{\bQ} 0$. 
Thus, we obtain $0 < (H \cdot f_{\ast}(\wC )) = a + \sum _{i=1}^rb_i - \frac{b_r}{\alpha _r} + c$ (see also Figure \ref{C} (a)). 
Furthermore, we also obtain $0 < a + \sum _{i=1}^rb_i - \frac{b_r}{\alpha _r}$ by the similar argument replacing $\wDelta ^{(1)}$ to $\wDelta ^{(1)} + \sum _{\mu =1}^{\beta}\mu \wD _{r+1,\mu} -\wE _{r+1}'$. 

Assume $\beta \ge 2$. 
Then we obtain $\dim |\wDelta ^{(2)}| \ge 0$ and $\dim |\wDelta ^{(3)}| \ge 0$ by Lemma \ref{linear system} (2) and (3) combined with $n+2=\alpha + \beta$ and $\beta \ge 2$. 
Hence, $|\wDelta ^{(2)}| \not= \emptyset$ and $|\wDelta ^{(3)}| \not= \emptyset$, so that we can take a general member $\wC'$ and $\wC''$ of $|\wDelta ^{(2)}|$ and $|\wDelta ^{(3)}|$, respectively. 
Notice $f_{\ast}(C') \sim _{\bQ} nF - \sum _{i=1}^r\alpha _iE_i - E_{r+1}' \not\sim _{\bQ} 0$ and $f_{\ast}(C'') \sim _{\bQ} (\beta -1)nF - \sum _{i=1}^r(\beta -1)\alpha _iE_i -(\beta -2)\beta E_{r+1} - E_{r+1}' \not\sim _{\bQ} 0$. 
Thus, we obtain $0 < (H \cdot f_{\ast}(\wC ')) = a + \sum _{i=1}^rb_i$ and $0 < \frac{1}{\beta -1}(H \cdot f_{\ast}(\wC '')) = a + \sum _{i=1}^rb_i + \frac{\beta -2}{\beta -1}c$ (see also Figure \ref{C} (b) and (c)). 
\end{proof}
Without less of generality, we may assume $\frac{b_1}{\alpha _1} \le \frac{b_2}{\alpha _2} \le \dots \le \frac{b_r}{\alpha _r}$ if $r>0$. 
Let $\wGamma$ be the same as in Lemma \ref{rational curve} (1). 
Then we notice that the configuration of $\wS$ is as in Figure \ref{Gamma} (a). 
Put $\Gamma := f_{\ast}(\wGamma) \sim _{\bQ} nF - \sum _{i=1}^r\alpha _iE_i - (\beta -1)E_{r+1}$. 
Then the following assertion holds: 
\begin{lem}\label{lem(4-2-2)}
With the same notations and the assumptions as above, we obtain $H \in \Ampc (S)$. 
\end{lem}
In order to show Lemma \ref{lem(4-2-2)}, we shall deal with two special cases in Lemmas \ref{lem(4-2-3)} and \ref{lem(4-2-4)}: 
\begin{lem}\label{lem(4-2-3)}
If $r=0$, then $H \in \Ampc (S)$. 
\end{lem}
\begin{proof}
By the assumption $r = 0$, we notice $\beta = n+2$. 
If $c \ge 0$, then we know $H \in \Ampc (S)$ by Lemma \ref{lem(4-1-2)} combined with Lemma \ref{lem(4-2-1)} (2). 
From now on, we thus assume $c<0$. 
Since $\beta = n+2$, we note $\Gamma \sim _{\bQ} nF - (n+1)E_{r+1}$. 
By noting Lemma \ref{lem(4-2-1)} (3), take the effective $\bQ$-divisor:
\begin{align*}
D := -\frac{c}{n+1} \Gamma + \left( a + \frac{n}{n+1}c\right)(E_1 + E_1')
\end{align*}
on $S$. 
Then $D \sim _{\bQ} H$. 
Moreover, by using Lemma \ref{lem(2-3)} (2) we know: 
\begin{align*}
S \backslash \Supp (D) &\simeq \wS \backslash \Supp (\wGamma + \wD _0 + \wF _1) \simeq \bA ^1_{\bk} \times \bA ^1_{\ast ,\bk}. 
\end{align*}
Thus, $H \in \Ampc (S)$. 
\end{proof}
\begin{lem}\label{lem(4-2-4)}
If $r>0$, $\beta \ge 2$ and $\frac{b_r}{\alpha _r} < \frac{c}{\beta -1}$, then $H \in \Ampc (S)$. 
\end{lem}
\begin{proof}
If $c \ge 0$, then we know $H \in \Ampc (S)$ by Lemma \ref{lem(4-1-2)} combined with Lemma \ref{lem(4-2-1)} (2). 
From now on, we thus assume $c<0$. 
By noting Lemma \ref{lem(4-2-1)} (3) and the assumption, take the effective $\bQ$-divisor:
\begin{align*}
D := &-\frac{c}{\beta -1} \Gamma + \sum _{i=1} ^r\alpha _i \left( \frac{c}{\beta -1} - \frac{b_i}{\alpha _i} \right) E_i' + \left( a + \sum _{i=1}^rb_i + \frac{\beta -2}{\beta -1}c\right)(E_{r+1} + E_{r+1}') 
\end{align*}
on $S$. 
Then $D \sim _{\bQ} H$. 
Moreover, by using Lemma \ref{lem(2-3)} (2) we know: 
\begin{align*}
S \backslash \Supp (D) &\simeq \wS \backslash \Supp \left(\wGamma + \wD _0 + \sum _{i=1}^r(\wF _i - \wE _i) + \wF _{r+1}\right)
\simeq \bA ^1_{\bk} \times \bA ^1_{\ast ,\bk}. 
\end{align*}
Thus, $H \in \Ampc (S)$.
\end{proof}
\begin{proof}[Proof of Lemma \ref{lem(4-2-2)}]
By Lemmas \ref{lem(4-2-3)} and \ref{lem(4-2-4)}, we can assume $r>0$ and further that either $\beta = 1$ or $\beta \ge 2$ and $\frac{b_r}{\alpha _r} \ge \frac{c}{\beta -1}$. 
For simplicity, we set $d := a + \sum _{i=1}^rb_i - \frac{b_r}{\alpha _r} + c$ and $d' := a + \sum _{i=1}^rb_i + (\beta -2) \cdot \frac{b_r}{\alpha _r}$. 
Note that $d >0$ and $d'>0$. 
Indeed, the former follows from Lemma \ref{lem(4-2-1)} (1), and the latter follows from Lemma \ref{lem(4-2-1)} (1) (resp. (3)) when $\beta = 1$ (resp. $\beta \ge 2$ and $\frac{b_r}{\alpha _r} \ge \frac{c}{\beta -1}$). 

Suppose that $\frac{b_1}{\alpha _1} = \frac{b_r}{\alpha _r}$. 
Then we know $n \cdot \frac{b_r}{\alpha _r} = \sum _{i=1}^rb_i + (\beta -2) \cdot \frac{b_r}{\alpha _r}$ because of $n+2=\alpha + \beta$ and $\frac{b_i}{\alpha _i} = \frac{b_r}{\alpha _r}$ for $i=1,\dots ,r$. 
Note $n - \beta + 1 \ge 0$ by virtue of $\alpha + \beta = n+2$ and $\alpha \ge 1$. 
Letting $\varepsilon$ be a positive rational number satisfying: 
\begin{align*}
\varepsilon < 
\left\{ \begin{array}{ll}
\min \left\{ \frac{d}{n-\beta +1},\ \frac{d'}{n}\right\} & \text{if $n-\beta +1>0$} \\ 
\frac{d'}{n} & \text{if $n-\beta +1=0$} 
\end{array}\right. ,
\end{align*}
we take the effective $\bQ$-divisor:
\begin{align*}
D := \left( -\frac{b_r}{\alpha _r} + \varepsilon \right) \Gamma + \sum _{i=1} ^r\alpha _i\varepsilon E_i + \left\{ d - (n - \beta + 1)\varepsilon \right\} E_{r+1} + \left( d' - n\varepsilon \right) E_{r+1}'
\end{align*}
on $S$. 
Then $D \sim _{\bQ} H$. 
Moreover, by using Lemma \ref{lem(2-3)} (2) we know: 
\begin{align*}
S \backslash \Supp (D) &\simeq \wS \backslash \Supp \left(\wGamma + \wD _0 + \sum _{i=1}^r (\wF _i -\wE _i') + \wF _{r+1}\right) 
\simeq \bA ^1_{\bk} \times \bA ^1_{\ast ,\bk}. 
\end{align*}
Thus, $H \in \Ampc (S)$.

In what follows, we can assume that $\frac{b_1}{\alpha _1} < \frac{b_r}{\alpha _r}$. 
Put $r' := \max \left\{ i \in \{ 1,\dots ,r-1\} \, |\, \frac{b_i}{\alpha _i} < \frac{b_r}{\alpha _r}\right\}$ and $\alpha ' := \sum _{i=1}^{r'}\alpha _i$. 
Then we know $n \cdot \frac{b_r}{\alpha _r} + \sum _{i=1}^{r'} \alpha _i \left( \frac{b_i}{\alpha _i} - \frac{b_r}{\alpha _r} \right) = \sum _{i=1}^rb_i + (\beta -2) \cdot \frac{b_r}{\alpha _r}$ because of $n+2=\alpha + \beta$ and $\frac{b_i}{\alpha _i} = \frac{b_r}{\alpha _r}$ for $i=r'+1,\dots ,r$. 
Note $n-\alpha ' \ge n - \alpha ' - \beta + 1 \ge 0$ by virtue of $\alpha + \beta = n+2$ and $\alpha \ge \alpha '+1$. 
Letting $\varepsilon$ be a positive rational number satisfying: 
\begin{align*}
\varepsilon < 
\left\{ \begin{array}{ll}
\min \left\{ \frac{b_r}{\alpha _r} - \frac{b_1}{\alpha _1},\ \frac{d}{n-\alpha '-\beta +1},\ \frac{d'}{n-\alpha '}\right\} & \text{if $n-\alpha '-\beta +1>0$} \\
\min \left\{ \frac{b_r}{\alpha _r} - \frac{b_1}{\alpha _1},\ \frac{d'}{n-\alpha '}\right\} & \text{if $n-\alpha '>n-\alpha '-\beta +1= 0$} \\
\frac{b_r}{\alpha _r} - \frac{b_1}{\alpha _1} & \text{if $n-\alpha '= 0$} 
\end{array}\right. ,
\end{align*}
we take the effective $\bQ$-divisor:
\begin{align*}
D := &\left( -\frac{b_r}{\alpha _r} + \varepsilon \right) \Gamma + \sum _{i=1} ^{r'}\alpha _i \left( \frac{b_r}{\alpha _r} - \frac{b_i}{\alpha _i} - \varepsilon \right) E_i' + \sum _{j=r'+1} ^r\alpha _j\varepsilon E_j \\
&\quad+ \left\{ d - (n-\alpha '-\beta +1)\varepsilon \right\} E_{r+1} + \left\{ d' - (n-\alpha ')\varepsilon \right\} E_{r+1}'
\end{align*}
on $S$. 
Then $D \sim _{\bQ} H$. 
Moreover, by using Lemma \ref{lem(2-3)} (2) we know: 
\begin{align*}
S \backslash \Supp (D) \simeq \wS \backslash \Supp \left(\wGamma + \wD _0 + \sum _{i=1}^{r'} (\wF _i - \wE _i) + \sum _{j=r'+1}^{r} (\wF _j - \wE _j') + \wF _{r+1}\right) \simeq \bA ^1_{\bk} \times \bA ^1_{\ast ,\bk}. 
\end{align*}
Thus, $H \in \Ampc (S)$. 
\end{proof}
\subsection{Case $(s,t)=(0,1)$ and $n \ge 2$}\label{4-3}
In this subsection, we prove the case of $(s,t)=(0,1)$ in Lemma \ref{lem(4-2)} under the assumption that $n \ge 2$. 
We use the notation from \S \S \ref{3-1}. 
Assume that $(s,t)=(0,1)$ and $n \ge 2$. 
If $(-K_{\wS})^2 \ge 6-n$ or $\gamma \ge 4$, then we obtain $\Ampc (S) = \Amp (S)$ by Lemma \ref{lem(4-1)}. 
From now on, we assume further that $(-K_{\wS})^2 = 5-n$, which implies $\alpha + \gamma = n+3$, and $\gamma \in \{ 2,3\}$. 

Let $H \in \Amp (S)$. 
Since $\Amp (S) \subseteq \Cl (S)_{\bQ} = \bQ [F] \oplus \left( \bigoplus _{i=1}^r\bQ [E_i] \right)$, we can write:
\begin{align*}
H \sim _{\bQ} aF + \sum _{i=1}^rb_iE_i
\end{align*}
for some rational numbers $a,b_1,\dots ,b_r$. 
We note $b_i = -\alpha _i(H \cdot E_i)<0$ for $i=1,\dots ,r$. 
Then the following lemma holds: 
\begin{lem}\label{lem(4-3-1)}
$2a + \sum _{i=1}^r2b_i -(4-\gamma) \cdot \frac{b_r}{\alpha _r} > 0$. 
\end{lem}
\begin{proof}
We consider two cases $\gamma =2$ and $\gamma =3$ separately. 

Assume that $\gamma =2$. 
Let $\widetilde{\Delta}^{(4)}$ be the same as in Lemma \ref{linear system} (4). 
By Lemma \ref{linear system} (4), we notice $\dim |\widetilde{\Delta}^{(4)}| \ge 0$ because of $n+3=\alpha + 2$. 
Hence, $|\widetilde{\Delta}^{(4)}| \not= \emptyset$, so that we can take a general member $\wC$ of $|\widetilde{\Delta}^{(4)}|$. 
Note  $f_{\ast}(\wC ) \sim _{\bQ} nF -\sum _{i=1}^r\alpha _iE_i +E_r - E_{r+1} \not\sim _{\bQ} 0$. 
Thus, we obtain $0 < 2(H \cdot f_{\ast}(\wC )) = 2a + \sum _{i=1}^r2b_i - \frac{2b_r}{\alpha _r}$ (see also Figure \ref{C} (d)). 

Assume that $\gamma =3$. 
Let $\widetilde{\Delta}^{(5)}$ be the same as in Lemma \ref{linear system} (5). 
By Lemma \ref{linear system} (5), we notice $\dim |\widetilde{\Delta}^{(5)}| \ge 1$ because of $n+3=\alpha + 3$. 
Hence, $|\widetilde{\Delta}^{(5)}| \not= \emptyset$, so that we can take a general member $\wC$ of $|\widetilde{\Delta}^{(5)}|$. 
Note $f_{\ast}(\wC ) \sim _{\bQ} 2nF -\sum _{i=1}^r2\alpha _iE_i +E_r - 3E_{r+1} \not\sim _{\bQ} 0$. 
Thus, we obtain $0 < (H \cdot f_{\ast}(\wC )) = 2a + \sum _{i=1}^r2b_i - \frac{b_r}{\alpha _r}$ (see also Figure \ref{C} (e)). 
\end{proof}
For simplicity, we set $d := 2a +\sum _{i=1}^r2b_i -(4-\gamma ) \cdot \frac{b_r}{\alpha _r}$. 
Note that $d >0$ by Lemma \ref{lem(4-3-1)}. 
Let $\wGamma$ be the same as in Lemma \ref{rational curve} (2). 
Notice that the configuration of $\wS$ is as in Figure \ref{Gamma} (b) (resp. (c)) when $\gamma = 2$ (resp. $\gamma =3$). 
Put $\Gamma := f_{\ast}(\wGamma) \sim _{\bQ} nF - \sum _{i=1}^r\alpha _iE_i - (\gamma -2)E_{r+1}$. 
Then the following assertion holds: 
\begin{lem}\label{(4-3-2)}
With the same notations and the assumptions as above, we obtain $H \in \Ampc (S)$. 
\end{lem}
\begin{proof}
Without less of generality, we may assume $\frac{b_1}{\alpha _1} \le \frac{b_2}{\alpha _2} \le \dots \le \frac{b_r}{\alpha _r}$. 

Suppose that $\frac{b_1}{\alpha _1} = \frac{b_r}{\alpha _r}$. 
Then we know $2n \cdot \frac{b_r}{\alpha _r} - (\gamma -2) \cdot \frac{b_r}{\alpha _r} = \sum _{i=1}^r2b_i - (4-\gamma ) \cdot \frac{b_r}{\alpha _r}$ because of $n+3 = \alpha + \gamma$ and $\frac{b_i}{\alpha _i} = \frac{b_r}{\alpha _r}$ for $i=1,\dots ,r$. 
Note $2n - \gamma + 2 > 0$ by virtue of $n \ge 2$ and $\gamma \in \{ 2,3\}$. 
Letting $\varepsilon$ be a positive rational number satisfying $\varepsilon < \frac{d}{2n-\gamma +2}$, 
we take the effective $\bQ$-divisor:
\begin{align*}
D := &\left( -\frac{b_r}{\alpha _r} + \varepsilon \right) \Gamma + \sum _{i=1} ^r\alpha _i\varepsilon E_i + \left\{d -(2n -\gamma +2)\varepsilon \right\} E_{r+1}
\end{align*}
on $S$. 
Then $D \sim _{\bQ} H$. 
Moreover, by using Lemma \ref{lem(2-3)} (2) we know: 
\begin{align*}
S \backslash \Supp (D) \simeq \wS \backslash \left( \wGamma + \wD _0 + \sum _{i=1}^r (\wF _i - \wE _i') + \wF _{r+1}\right) 
\simeq \bA ^1_{\bk} \times \bA ^1_{\ast ,\bk}. 
\end{align*}
Thus, $H \in \Ampc (S)$.

In what follows, we can assume that $\frac{b_1}{\alpha _1} < \frac{b_r}{\alpha _r}$. 
Put $r' := \max \left\{ i \in \{ 1,\dots ,r-1\}\,|\, \frac{b_i}{\alpha _i}<\frac{b_r}{\alpha _r}\right\}$ and $\alpha ' := \sum _{i=1}^{r'}\alpha _i$. 
Then we know $2n \cdot \frac{b_r}{\alpha _r} + 2\sum _{i=1}^{r'}\alpha _i \left( \frac{b_i}{\alpha _i} - \frac{b_r}{\alpha _r} \right)  - (\gamma -2) \cdot \frac{b_r}{\alpha _r} = \sum _{i=1}^r2b_i - (4-\gamma ) \cdot \frac{b_r}{\alpha _r}$ because of $n+3 = \alpha + \gamma$ and $\frac{b_i}{\alpha _i} = \frac{b_r}{\alpha _r}$ for $i=r'+1,\dots ,r$. 
Note $2n -\gamma +2 -2\alpha '  \ge 0$ by virtue of $\alpha + \gamma = n + 3$ and $\alpha \ge \alpha '+1$. 
Letting $\varepsilon$ be a positive rational number satisfying: 
\begin{align*}
\varepsilon < 
\left\{ \begin{array}{ll}
\min \left\{ \frac{b_r}{\alpha _r} - \frac{b_1}{\alpha _1},\ \frac{d}{2n -\gamma +2 -2\alpha '} \right\} &\text{if $2n -\gamma +2 -2\alpha ' > 0$} \\
\frac{b_r}{\alpha _r} - \frac{b_1}{\alpha _1} &\text{if $2n -\gamma +2 -2\alpha ' = 0$} 
\end{array}\right. ,
\end{align*}
we take the effective $\bQ$-divisor:
\begin{align*}
D := &\left( -\frac{b_r}{\alpha _r} + \varepsilon \right) \Gamma + \sum _{i=1} ^{r'}\alpha _i \left( \frac{b_r}{\alpha _r} - \frac{b_i}{\alpha _i} - \varepsilon \right) E_i' + \sum _{j=r'+1} ^r\alpha _j\varepsilon E_j + \left\{d -(2n -2\alpha ' -\gamma +2)\varepsilon \right\} E_{r+1}
\end{align*}
on $S$. 
Then $D \sim _{\bQ} H$. 
Moreover, by using Lemma \ref{lem(2-3)} (2) we know: 
\begin{align*}
S \backslash \Supp (D) \simeq \wS \backslash \Supp \left(\wGamma + \wD _0 + \sum _{i=1}^{r'} (\wF _i - \wE _i) + \sum _{j=r'+1}^{r} (\wF _j - \wE _j') + \wF _{r+1}\right) \simeq \bA ^1_{\bk} \times \bA ^1_{\ast ,\bk}. 
\end{align*}
Thus, $H \in \Ampc (S)$. 
\end{proof}
\subsection{Case $(s,t)=(0,0)$ and $n \ge 2$}\label{4-4}
In this subsection, we prove the case of $(s,t)=(0,0)$ in Lemma \ref{lem(4-2)} under the assumption that $n \ge 2$. 
We use the notation from \S \S \ref{3-1}. 
Assume that $(s,t)=(0,0)$ and $n \ge 2$. 
If $(-K_{\wS})^2 \ge 7-n$, then we obtain $\Ampc (S) = \Amp (S)$ by Lemma \ref{lem(4-1)}. 
From now on, we assume further that $(-K_{\wS})^2 \in \{ 6-n,\ 5-n\}$; namely, $\alpha \in \{ n+2,\ n+3\}$. 

Let $H \in \Amp (S)$. 
Since $\Amp (S) \subseteq \Cl (S)_{\bQ} = \bQ [F] \oplus \left( \bigoplus _{i=1}^r\bQ [E_i] \right)$, we can write: 
\begin{align*}
H \sim _{\bQ} aF + \sum _{i=1}^rb_iE_i
\end{align*}
for some rational numbers $a,b_1,\dots ,b_r$. 
Here, we note $b_i = -\alpha _i(H \cdot E_i)<0$ for $i=1,\dots ,r$. 
Without less of generality, we may assume $\frac{b_1}{\alpha _1} \le \frac{b_2}{\alpha _2} \le \dots \le \frac{b_r}{\alpha _r}$. 
By noting $\alpha > n+1$, put $r' := \min \{ i \in \{1,\dots ,r\} \,|\, \alpha _1 + \dots + \alpha _i \ge n+1\}$ and $\alpha ' := \sum _{i=1}^{r'}\alpha _i$. 
Then the following lemma holds: 
\begin{lem}\label{lem(4-4-1)}
$a +\sum _{i=1}^{r'}b_i +(n+1-\alpha ')\cdot \frac{b_{r'}}{\alpha _{r'}} > 0$. 
\end{lem}
\begin{proof}
Let $\widetilde{\Delta}^{(6)}$ be the same as in Lemma \ref{linear system} (6). 
Then we obtain $\dim |\widetilde{\Delta}^{(6)}| \ge 0$ by Lemma \ref{linear system} (6), 
Hence, $|\widetilde{\Delta} ^{(6)}| \not= \emptyset$, so that we can take a general member $\wC$ of $|\widetilde{\Delta}^{(6)}|$. 
Notice $f_{\ast}(\wC ) \sim _{\bQ} nF - \sum _{i=1}^{r'}\alpha _iE_i + (n+1 - \alpha ')E_{r'} \not\sim _{\bQ} 0$. 
Thus, we obtain $0 < (H \cdot f_{\ast}(\wC )) = a +\sum _{i=1}^{r'}b_i +(n+1-\alpha ') \cdot \frac{b_{r'}}{\alpha _{r'}}$ (see also Figure \ref{C} (f)). 
\end{proof}
For simplicity, we set $d := a +\sum _{i=1}^{r'}b_i +(n+1-\alpha ') \cdot \frac{b_{r'}}{\alpha _{r'}}$. 
Note that $d>0$ by Lemma \ref{lem(4-4-1)}. 
Let $\wGamma$ be the same as in Lemma \ref{rational curve} (3). 
Notice that the configuration of $\wS$ is as in Figure \ref{Gamma} (d). 
Meanwhile, there exists a unique fiber $\wF _0$ of $g$ such that $\wGamma \cap \wD _0 \cap \wF _0 \not= \emptyset$ because $(\wGamma \cdot \wD _0) = 1$. 
Put $\Gamma := f_{\ast}(\wGamma ) \sim _{\bQ} (n+1)F - \sum _{i=1}^r\alpha _iE_i$ and $F_0 := f_{\ast}(\wF _0)$. 
Then the following assertion holds: 
\begin{lem}\label{lem(4-4-2)}
With the same notations and the assumptions as above, we obtain $H \in \Ampc (S)$. 
\end{lem}
\begin{proof}
Suppose that $\frac{b_1}{\alpha _1} = \frac{b_{r'}}{\alpha _{r'}}$. 
Then we know $(n+1) \cdot \frac{b_{r'}}{\alpha _{r'}} = \sum _{i=1}^{r'}b_i + (n+1-\alpha ') \cdot \frac{b_{r'}}{\alpha _{r'}}$ because of $\frac{b_i}{\alpha _i} = \frac{b_{r'}}{\alpha _{r'}}$ for $i=1,\dots ,r'$. 
Letting $\varepsilon$ be a positive rational number satisfying $\varepsilon < \frac{d}{n+1}$, 
we take the effective $\bQ$-divisor:
\begin{align*}
D := \left( -\frac{b_{r'}}{\alpha _{r'}} + \varepsilon \right) \Gamma + \left\{ d -(n+1)\varepsilon \right\} F_0 + \sum _{i=1} ^r\alpha _i\left( \frac{b_i}{\alpha _i} - \frac{b_{r'}}{\alpha _{r'}} + \varepsilon \right) E_i 
\end{align*}
on $S$. 
Then $D \sim _{\bQ} H$. 
Moreover, by using Lemma \ref{lem(2-3)} (3) we know: 
\begin{align*}
S \backslash \Supp (D) \simeq \wS \backslash \Supp \left( \wGamma + \wD _0 + \wF _0 + \sum _{i=1}^r (\wF _i - \wE _i') \right) 
\simeq \bA ^1_{\bk} \times \bA ^1_{\ast ,\bk}. 
\end{align*}
Thus, $H \in \Ampc (S)$.

In what follows, we can assume that $\frac{b_1}{\alpha _1} < \frac{b_{r'}}{\alpha _{r'}}$. 
Put $r'' := \max \left\{ i \in \{ 1,\dots ,r'-1\}\,|\, \frac{b_i}{\alpha _i}<\frac{b_{r'}}{\alpha _{r'}}\right\}$ and $\alpha '' := \sum _{i=1}^{r''}\alpha _i$. 
Then we know $(n+1) \cdot \frac{b_{r'}}{\alpha _{r'}} + \sum _{i=1}^{r''}\alpha _i \left( \frac{b_i}{\alpha _i} - \frac{b_{r'}}{\alpha _{r'}} \right) = \sum _{i=1}^{r'}b_i + (n+1-\alpha ') \cdot \frac{b_{r'}}{\alpha _{r'}}$ because of $\frac{b_i}{\alpha _i} = \frac{b_{r'}}{\alpha _{r'}}$ for $i=r''+1,\dots ,r'$. 
Note $n+1-\alpha '' \ge 1$ by virtue of $n+1 \ge \alpha -1 \ge \alpha ' \ge \alpha ''+1$. 
Letting $\varepsilon$ be a positive rational number satisfying $\varepsilon < \min \left\{ \frac{b_{r'}}{\alpha _{r'}} - \frac{b_1}{\alpha _1},\ \frac{d}{n+1-\alpha ''}\right\}$, 
we take the effective $\bQ$-divisor:
\begin{align*}
D := &\left( -\frac{b_{r'}}{\alpha _{r'}} + \varepsilon \right) \Gamma + \left\{ d -(n+1-\alpha '')\varepsilon \right\} F_0 \\
&\quad+ \sum _{i=1} ^{r''}\alpha _i\left( \frac{b_{r'}}{\alpha _{r'}} - \frac{b_i}{\alpha _i} - \varepsilon \right) E_i'
+ \sum _{j=r''+1} ^r\alpha _j\left( \frac{b_j}{\alpha _j} - \frac{b_{r'}}{\alpha _{r'}} + \varepsilon \right) E_j
\end{align*}
on $S$. 
Then $D \sim _{\bQ} H$. 
Moreover, by using Lemma \ref{lem(2-3)} (3) we know: 
\begin{align*}
S \backslash \Supp (D) \simeq \wS \backslash \Supp \left( \wGamma + \wD _0 + \wF _0 + \sum _{i=1}^{r''} (\wF _i - \wE _i) + \sum _{j=r''+1} ^r (\wF _j - \wE _j') \right) \simeq \bA ^1_{\bk} \times \bA ^1_{\ast ,\bk}. 
\end{align*}
Thus, $H \in \Ampc (S)$. 
\end{proof}
\subsection{Case $s+t \le 1$ and $n \le 1$}\label{4-5}
In this subsection, we prove Lemma \ref{lem(4-2)} under the assumption that $n \le 1$. 
We use the notation from \S \S \ref{3-1}. 

Assume that $(s,t)=(1,0)$. 
If $(-K_{\wS})^2 \ge 6-n$ or $\beta ' \ge 2$, then we obtain $\Ampc (S) = \Amp (S)$ by Lemma \ref{lem(4-1)} because $(-K_{\wS})^2 + \beta ' \ge 7-n$. 
From now on, we assume further that $(-K_{\wS})^2 = 5-n$ and $\beta ' = 1$; namely, $\alpha + \beta = n+2$. 
Then $\wS$ is a weak del Pezzo surface whose type is one of the following according to $n$: 
\begin{itemize}
\item $n=0$: $A_2$ and $A_1$; 
\item $n=1$: $A_3$ (with $4$ lines), $A_2$, $2A_1$ (with $8$ lines) and $A_1$. 
\end{itemize}
Notice that $\wS$ satisfies the condition $(\ast _2)$ except for ${\rm Dyn}(S) \not= 2A_1$ by Table \ref{table}. 
That is, if ${\rm Dyn}(S) \not= 2A_1$, we obtain $\Ampc (S) = \Amp (S)$ because Lemmas \ref{lem(4-1)} and \ref{lem(4-2)} are shown for $n \ge 2$. 
Meanwhile, if ${\rm Dyn}(S) = 2A_1$, we also obtain $\Ampc (S) = \Amp (S)$ by the following lemma: 
\begin{lem}
Assume that $(-K_{\wS})^2=4$ and $S$ is of type $2A_1$ (with $8$ lines). 
Then $\Ampc (S) = \Amp (S)$. 
\end{lem}
\begin{proof}
By Lemma \ref{lem(2-1)}, there exist two birational morphisms $h_1:\wS \to \hS$ and $h_2:\hS \to \bP ^2_{\bk}$ such that $h_1$ is a blowing-up at a point $x_5$ and $h_2$ is a blowing-up at four points $x_1,\dots ,x_4$. 
Let $\hE _1,\dots ,\hE _4$ be irreducible components of the reduced exceptional divisor of $h_2$ such that $\hE _i = h_2^{-1}(x_i)$ for $i=1,\dots ,4$. 
By {\cite[Proposition 6.1]{CT88}}, we can assume that $x_5 \in \hE _4$ and there exists a line $L_{1,2,3}$ on $\bP ^2_{\bk}$ such that $x_1,x_2,x_3 \in L_{1,2,3}$ and $x_4 \not\in L_{1,2,3}$. 
Set $h := h_2 \circ h_1$, let $e_0$ be a total transform of a general line on $\bP ^2_{\bk}$ by $h$, let $e_5$ be the reduced exceptional divisor of $h_1$, and let $e_1,\dots ,e_4$ be total transforms on $\wS$ of $\hE _1,\dots ,\hE _4$, respectively. 
By construction of $h$, $\wS$ has exactly two $(-2)$-curves $\wD _0$ and $\wD _{\infty}$ such that: 
\begin{align*}
\wD _0 \sim e_0 - e_1 - e_2 - e_3,\ \wD _{\infty} \sim e_4 - e_5. 
\end{align*}
Moreover, by Lemma \ref{lem(2-2)} there exist $(-1)$-curves $\wE _1$, $\wE _2$, $\wE _3$, $\wE _4$, $\wE _1'$, $\wE _2'$, $\wE _3'$ and $\wE _4'$ on $\wS$ such that: 
\begin{align*}
\wE _1 \sim e_0 - e_1 - e_4,\ \wE _2 \sim e_0 - e_2 - e_4,\ \wE _3 \sim e_0 - e_3 - e_4,\ \wE _4 \sim e_5,\\
\wE _1' \sim e_1,\ \wE _2' \sim e_2,\ \wE _3' \sim e_3,\ \wE _4' \sim e_0 - e_4 - e_5. 
\end{align*}
Then the divisor $e_0-e_4$ defines a $\bP ^1$-fibration $g := \Phi _{|e_0-e_4|}:\wS \to \bP ^1_{\bk}$. 
Notice that $g$ has exactly four singular fibers $\wF _i := \wE _i' + \wE _i$ for $i=1,\dots ,4$; moreover, $\wD _0$ and $\wD _{\infty}$ are sections of $g$ (see also Figure \ref{fig(4-1)}). 
Let $\wF$ be a general fiber of $g$. 
Put $F := f_{\ast}(\wF )$, $E_i := f_{\ast}(\wE _i)$ and $E_i' := f_{\ast}(\wE _i')$ for $i=1,\dots ,4$. 
Since $\wD _{\infty} \sim \wD _0 + 2\wF - \sum _{i=1}^4\wE _i$ and $f_{\ast}(\wD _{\infty}) \sim  _{\bQ} 0$, we have $2F \sim _{\bQ} \sum _{i=1}^4E_i$. 
Hence, $\Cl (S) _{\bQ} = \bigoplus _{i=1}^4\bQ [E_i]$ because of $\Cl (\wS )_{\bQ} = \bigoplus _{i=0}^5 \bQ [e_i] = \bQ [\wD _0] \oplus \bQ [\wF ] \oplus \left( \bigoplus _{i=1}^4 \bQ [\wE _i] \right)$. 

Let $H \in \Amp (S)$. 
Since $\Amp (S) \subseteq \Cl (S)_{\bQ} = \bigoplus _{i=1}^4\bQ [E_i]$, we can write: 
\begin{align*}
H \sim _{\bQ} \sum _{i=1}^4a_iE_i
\end{align*}
for some rational numbers $a_1,\dots ,a_4$. 
Since $(E_i' \cdot E_j) = \delta _{i,j}$, we know $a_i = (H \cdot E_i') > 0$ for $i=1,\dots ,4$. 
Letting $\varepsilon$ be a positive rational number satisfying $\varepsilon < \min \{ a_1,\dots ,a_4\}$, 
we take the effective $\bQ$-divisor:
\begin{align*}
D := 2\varepsilon F + \sum _{i=1}^4(a_i-\varepsilon )E_i
\end{align*}
on $S$. 
Then $D \sim _{\bQ} H$. 
Moreover, by using Lemma \ref{lem(2-3)} (2) we know: 
\begin{align*}
S \backslash \Supp (D) &\simeq \wS \backslash \Supp \left( \wD _0 + \wD _{\infty} + \wF + \sum _{i=1}^4\wE _i \right) \simeq \bA ^1_{\bk} \times \bA ^1_{\ast ,\bk}. 
\end{align*}
Thus, $H \in \Ampc (S)$; namely, we obtain $\Ampc (S) = \Amp (S)$. 
\end{proof}
\begin{figure}[t]
\begin{center}\scalebox{0.5}{\begin{tikzpicture}
\draw [very thick] (-1,0.5)--(8,0.5);
\draw [very thick] (-1,4.5)--(8,4.5);
\node at (-1.4,0.5) {\Large $\wD _0$};
\node at (-1.4,4.5) {\Large $\wD _{\infty}$};

\draw [dashed] (1,0)--(0,3);
\draw [dashed,thick] (0,2)--(1,5);
\node at (1,-0.4) {\Large $\wE _1'$};
\node at (1,5.4) {\Large $\wE _1$};

\draw [dashed] (3,0)--(2,3);
\draw [dashed,thick] (2,2)--(3,5);
\node at (3,-0.4) {\Large $\wE _2'$};
\node at (3,5.4) {\Large $\wE _2$};

\draw [dashed] (5,0)--(4,3);
\draw [dashed,thick] (4,2)--(5,5);
\node at (5,-0.4) {\Large $\wE _3'$};
\node at (5,5.4) {\Large $\wE _3$};

\draw [dashed] (7,0)--(6,3);
\draw [dashed,thick] (6,2)--(7,5);
\node at (7,-0.4) {\Large $\wE _4'$};
\node at (7,5.4) {\Large $\wE _4$};
\end{tikzpicture}}\end{center}
\caption{Configuration of $\wS$ with $(d,{\rm Dyn}(S))=(4,2A_1)$ with $8$ lines}\label{fig(4-1)}
\end{figure}

Assume that $(s,t)=(0,1)$. 
If $(-K_{\wS})^2 \ge 6-n$ or $\gamma \ge 4$, then we obtain $\Ampc (S) = \Amp (S)$ by Lemma \ref{lem(4-1)}. 
From now on, we assume further that $(-K_{\wS})^2 = 5-n$, which implies $\alpha + \gamma = n+3$, and $\gamma \in \{ 2,3\}$. 
Then $\wS$ is a weak del Pezzo surface whose type is one of the following according to $n$: 
\begin{itemize}
\item $n=0$: $A_3$ and $2A_1$; 
\item $n=1$: $D_4$, $A_3$ (with $5$ lines), $3A_1$ and $2A_1$ (with $9$ lines). 
\end{itemize}
Hence, $\wS$ satisfies the condition $(\ast _2)$ by Table \ref{table}, so that we obtain $\Ampc (S) = \Amp (S)$ because Lemmas \ref{lem(4-1)} and \ref{lem(4-2)} are shown for $n \ge 2$. 

Thus, we can assume that $(s,t)=(0,0)$ in what follows. 
If $(-K_{\wS})^2 \ge 7-n$, then we obtain $\Ampc (S) = \Amp (S)$ by Lemma \ref{lem(4-1)}. 
From now on, we assume further that $(-K_{\wS})^2 \in \{ 6-n,\ 5-n\}$; namely, $\alpha \in \{ n+2,\ n+3\}$. 

Let $H \in \Amp (S)$. 
Since $\Amp (S) \subseteq \Cl (S)_{\bQ} = \bQ [D_0] \oplus \bQ [F] \oplus \left( \bigoplus _{i=1}^r\bQ [E_i] \right)$, we can write:
\begin{align*}
H \sim _{\bQ} a_0D_0 + aF + \sum _{i=1}^rb_iE_i
\end{align*}
for some rational numbers $a_0,a,b_1,\dots ,b_r$. 
Here, we note $b_i = -\alpha _i(H \cdot E_i)<0$ for $i=1,\dots ,r$. 
Without less of generality, we may assume $\frac{b_1}{\alpha _1} \le \frac{b_2}{\alpha _2} \le \dots \le \frac{b_r}{\alpha _r}$. 
By noting $\alpha > n+1$, put $r' := \min \{ i \in \{1,\dots ,r\} \,|\, \alpha _1 + \dots + \alpha _i \ge n+1\}$ and $\alpha ' := \sum _{i=1}^{r'}\alpha _i$. 
For simplicity, we set $d_0 := a_0 +\frac{b_{r'}}{\alpha _{r'}}$ and $d := a +\sum _{i=1}^{r'-1}b_i +(n+1-\alpha ') \cdot \frac{b_{r'}}{\alpha _{r'}}$. 
We notice $d_0>0$ and $d >0$. 
Indeed, the former follows from $d_0 = (H \cdot E_{r'}') > 0$, and the latter can be shown by the similar argument to Lemma \ref{lem(4-4-1)}. 
Let $\wGamma$ be the same as in Lemma \ref{rational curve} (3). 
Notice that the configuration of $\wS$ is as in Figure \ref{Gamma} (d). 
Meanwhile, there exists a unique fiber $\wF _0$ of $g$ such that $\wGamma \cap \wD _0 \cap \wF _0 \not= \emptyset$ because $(\wGamma \cdot \wD _0) = 1$. 
Put $\Gamma := f_{\ast}(\wGamma ) \sim _{\bQ} D_0 + (n+1)F - \sum _{i=1}^r\alpha _iE_i$ and $F_0 := f_{\ast}(\wF _0)$. 
Then the following assertion holds: 
\begin{lem}\label{lem(4-5-2)}
With the same notations and the assumptions as above, we obtain $H \in \Ampc (S)$. 
\end{lem}
\begin{proof}
Suppose that $\frac{b_1}{\alpha _1} = \frac{b_{r'}}{\alpha _{r'}}$. 
Then we know $(n+1) \cdot \frac{b_{r'}}{\alpha _{r'}} = \sum _{i=1}^{r'}b_i + (n+1-\alpha ') \cdot \frac{b_{r'}}{\alpha _{r'}}$ because of $\frac{b_i}{\alpha _i} = \frac{b_{r'}}{\alpha _{r'}}$ for $i=1,\dots ,r'$. 
Letting $\varepsilon$ be a positive rational number satisfying $\varepsilon < \min \left\{ d_0,\ \frac{d}{n+1}\right\}$, 
we take the effective $\bQ$-divisor:
\begin{align*}
D := \left( -\frac{b_{r'}}{\alpha _{r'}} + \varepsilon \right) \Gamma + (d_0-\varepsilon )D_0 + \left\{ d -(n+1)\varepsilon \right\} F_0 + \sum _{i=1} ^r\alpha _i\left( \frac{b_i}{\alpha _i} - \frac{b_{r'}}{\alpha _{r'}} + \varepsilon \right) E_i 
\end{align*}
on $S$. 
Then $D \sim _{\bQ} H$. 
Moreover, by using Lemma \ref{lem(2-3)} (3) we know: 
\begin{align*}
S \backslash \Supp (D) \simeq \wS \backslash \Supp \left( \wGamma + \wF _0 + \wD _0 + \sum _{i=1}^r (\wF _i -\wE _i') \right) 
\simeq \bA ^1_{\bk} \times \bA ^1_{\ast ,\bk}. 
\end{align*}
Thus, $H \in \Ampc (S)$.

In what follows, we can assume that $\frac{b_1}{\alpha _1} < \frac{b_{r'}}{\alpha _{r'}}$. 
Put $r'' := \max \left\{ i \in \{ 1,\dots ,r'-1\}\,|\, \frac{b_i}{\alpha _i}<\frac{b_{r'}}{\alpha _{r'}}\right\}$ and $\alpha '' := \sum _{i=1}^{r''}\alpha _i$. 
Then we know $(n+1) \cdot \frac{b_{r'}}{\alpha _{r'}} + \sum _{i=1}^{r''}\alpha _i \left( \frac{b_i}{\alpha _i} - \frac{b_{r'}}{\alpha _{r'}} \right) = \sum _{i=1}^{r'}b_i + (n+1-\alpha ') \cdot \frac{b_{r'}}{\alpha _{r'}}$ because of $\frac{b_i}{\alpha _i} = \frac{b_{r'}}{\alpha _{r'}}$ for $i=r''+1,\dots ,r'$. 
Note $n+1-\alpha '' \ge 1$ by virtue of $n+1 \ge \alpha -1 \ge \alpha ' \ge \alpha ''+1$. 
Letting $\varepsilon$ be a positive rational number satisfying $\varepsilon < \min \left\{ \frac{b_{r'}}{\alpha _{r'}} - \frac{b_1}{\alpha _1},\ d_0,\ \frac{d}{n+1-\alpha ''}\right\}$, 
we take the effective $\bQ$-divisor:
\begin{align*}
D := &\left( -\frac{b_{r'}}{\alpha _{r'}} + \varepsilon \right) \Gamma + (d_0 - \varepsilon )D_0 + \left\{ d -(n+1-\alpha '')\varepsilon \right\} F_0 \\
&\quad+ \sum _{i=1} ^{r''}\alpha _i\left( \frac{b_{r'}}{\alpha _{r'}} - \frac{b_i}{\alpha _i} - \varepsilon \right) E_i'
+ \sum _{j=r''+1} ^r\alpha _j\left( \frac{b_j}{\alpha _j} - \frac{b_{r'}}{\alpha _{r'}} + \varepsilon \right) E_j
\end{align*}
on $S$. 
Then $D \sim _{\bQ} H$. 
Moreover, by using Lemma \ref{lem(2-3)} (3) we know: 
\begin{align*}
S \backslash \Supp (D) \simeq \wS \backslash \left( \wGamma + \wD _0 + \wF _0 + \sum _{i=1}^{r''} (\wF _i -\wE _i) + \sum _{j=r''+1}^r (\wF _j - \wE _j') \right) \simeq \bA ^1_{\bk} \times \bA ^1_{\ast ,\bk}. 
\end{align*}
Thus, $H \in \Ampc (S)$. 
\end{proof}
\begin{rem}
Lemma \ref{lem(4-5-2)} presents the proof of Theorem \ref{CPW17} (2) (except for the case $d=9$) in a different way from {\cite{CPW17}}. 
\end{rem}
By the arguments in \S \S \ref{4-2}, \S \S \ref{4-3}, \S \S \ref{4-4} and \S \S \ref{4-5}, the proof of Lemma \ref{lem(4-2)} is thus completed. 
\section{Proof of Theorem \ref{main(1)}}\label{5}
In this section, we prove Theorem \ref{main(1)}. 
Let $S$ be a Du Val del Pezzo surface of degree $d \ge 3$ such that $\Sing (S) \not= \emptyset$. 
The proof of this theorem follows from Proposition \ref{prop(3)} and Theorem \ref{thm(4)} except for only one case. 
Hence, we treat this case in the following lemma: 
\begin{lem}\label{lem(5-1)}
Assume that $d=3$ and $S$ is of type $4A_1$. 
Then $\Ampc (S) = \Amp (S)$. 
\end{lem}
\begin{proof}
Let $f:\wS \to S$ be the minimal resolution. 
By {\cite[\S \S 9.2]{Dol12}}, there exists a blowing-up $h:\wS \to \bP ^2_{\bk}$ at six points $x_1,\dots ,x_6$ on the projective plane $\bP ^2_{\bk}$, which are intersection points of four lines in a general linear position on $\bP ^2_{\bk}$. 
By symmetry, we can assume that these four lines are $L_{1,2,6}$, $L_{1,3,4}$, $L_{2,3,5}$ and $L_{4,5,6}$ on $\bP ^2_{\bk}$, where $L_{i_1,i_2,i_3}$ means the line passing through $x_{i_1}$, $x_{i_2}$ and $x_{i_3}$. 
Let $e_0$ be a total transform of a general line on $\bP ^2_{\bk}$ by $h$ and let $e_1,\dots ,e_6$ be irreducible components of the exceptional divisor of $h$ such that $e_i = h^{-1}(x_i)$ for $i=1,\dots ,6$. 
By construction of $h$, $\wS$ has exactly four $(-2)$-curves $\wD _0$, $\wD _1$, $\wD _2$ and $\wD _{\infty}$ such that: 
\begin{align*}
\wD _0 \sim e_0 - e_4 - e_5 - e_6,\ \wD _1 \sim e_0 - e_1 - e_3 - e_4,\\
\wD _2 \sim e_0 - e_2 - e_3 - e_5,\ \wD _{\infty} \sim e_0 - e_1 - e_2 - e_6. 
\end{align*}
Moreover, by Lemma \ref{lem(2-2)} there exist $(-1)$-curves $\wE _1$, $\wE _2$, $\wE _3$, $\wE _1'$, $\wE _2'$, $\wE _3'$, $\wC _1$ and $\wC _2$ on $\wS$ such that: 
\begin{align*}
\wE _1 \sim e_1,\ \wE _2 \sim e_2,\ \wE _3 \sim e_0 - e_3 - e_6,\ \wE _1' \sim e_4,\ \wE _2' \sim e_5,\ \wE _3' \sim e_6,\\ \wC _1 \sim e_0 -e_1 - e_5,\ \wC _2 \sim e_0 -e_2 - e_4. 
\end{align*}
Then the divisor $e_0-e_3$ defines a $\bP ^1$-fibration $g := \Phi _{|e_0-e_3|}:\wS \to \bP ^1_{\bk}$. 
Notice that $g$ has exactly three singular fibers $\wF _1 := \wE _1' + \wD _1 + \wE _1$, $\wF _2 := \wE _2' + \wD _2 + \wE _2$ and $\wF _3 := \wE _3' + \wE _3$; moreover, $\wD _0$ and $\wD _{\infty}$ are sections of $g$ (see also Figure \ref{fig(5-1)}). 
Let $\wF$ be a general fiber of $g$. 
Put $F := f_{\ast}(\wF )$ and $E_i := f_{\ast}(\wE _i)$ for $i=1,2,3$. 
Since $\wD _{\infty} \sim \wD _0 + 2\wF - \sum _{i=1}^2(\wD _i + 2\wE _i)$ and $f_{\ast}(\wD _{\infty}) \sim  _{\bQ} 0$, we have $F \sim _{\bQ} E_1 + E_2$. 
Hence, $\Cl (S) _{\bQ} = \bigoplus _{i=1}^3\bQ [E_i]$ because of $\Cl (\wS )_{\bQ} = \bigoplus _{i=0}^6 \bQ [e_i] = \bQ [\wD _0] \oplus \bQ [\wF ] \oplus \left( \bigoplus _{i=1}^2 \bQ [\wD _i] \oplus \bQ [\wE _i] \right) \oplus \bQ [\wE _3]$. 

Let $H \in \Amp (S)$. 
Since $\Amp (S) \subseteq \Cl (S)_{\bQ} = \bigoplus _{i=1}^3\bQ [E_i]$, we can write: 
\begin{align*}
H \sim _{\bQ} \sum _{i=1}^3a_iE_i
\end{align*}
for some rational numbers $a_1,a_2,a_3$. 
For $i=1,2$, since $f_{\ast}(\wC _i) \sim _{\bQ} 2F - 2E_i -E_3 \not\sim _{\bQ} 0$, we obtain $a_i + a_3 = (H \cdot f_{\ast}(\wC _i)) > 0$. 
Meanwhile, we know $a_3 = -(H \cdot E_3) < 0. $
Letting $\varepsilon$ be a positive rational number satisfying $\varepsilon < \min \{ a_1+a_3,\ a_2+a_3\}$, 
we take the effective $\bQ$-divisor:
\begin{align*}
D := (a_1+a_3 -\varepsilon )E_1 + (a_2+a_3 -\varepsilon )E_2 + \varepsilon E_3 + (-a_3 +\varepsilon )E_3'
\end{align*}
on $S$. 
Then $D \sim _{\bQ} H$. 
Moreover, by using Lemma \ref{lem(2-3)} (2) we know: 
\begin{align*}
S \backslash \Supp (D) &\simeq \wS \backslash \Supp \left( \wD _0 + \wD _{\infty} + \sum _{i=1}^2(\wF _i - \wE _i') + \wF _3 \right) \simeq \bA ^1_{\bk} \times \bA ^1_{\ast ,\bk}. 
\end{align*}
Thus, $H \in \Ampc (S)$; namely, we obtain $\Ampc (S) = \Amp (S)$. 
\end{proof}
\begin{figure}[t]
\begin{center}\scalebox{0.5}{\begin{tikzpicture}
\draw [very thick] (-1,0.5)--(8,0.5);
\draw [very thick] (-1,4.5)--(4,4.5);
\draw [very thick] (6,1.5)--(8,1.5);
\draw [very thick] (4,4.5) .. controls (5,4.5) and (5,1.5) .. (6,1.5);
\node at (-1.4,0.5) {\Large $\wD _0$};
\node at (-1.4,4.5) {\Large $\wD _{\infty}$};

\draw [dashed] (1,0)--(0,2);
\draw [very thick] (0.2,1.25)--(0.2,3.75);
\draw [dashed,thick] (0,3)--(1,5);
\node at (1,-0.4) {\Large $\wE _1'$};
\node at (0.6,2.5) {\Large $\wD _1$};
\node at (1,5.4) {\Large $\wE _1$};

\draw [dashed] (4,0)--(3,2);
\draw [very thick] (3.2,1.25)--(3.2,3.75);
\draw [dashed,thick] (3,3)--(4,5);
\node at (4,-0.4) {\Large $\wE _2'$};
\node at (3.6,2.5) {\Large $\wD _2$};
\node at (4,5.4) {\Large $\wE _2$};

\draw [dashed,thick] (6,0)--(7,3);
\draw [dashed,thick] (7,2)--(6,5);
\node at (6,-0.4) {\Large $\wE _3'$};
\node at (6,5.4) {\Large $\wE _3$};
\end{tikzpicture}}\end{center}
\caption{Configuration of $\wS$ with $(d,{\rm Dyn}(S))=(3,4A_1)$}\label{fig(5-1)}
\end{figure}
Thus, we can show Theorem \ref{main(1)} as follows: 
\begin{proof}[Proof of Theorem \ref{main(1)}]
If $\rho (S)=1$, then we know $\Ampc (S) = \Amp (S) = \bQ _{>0}[-K_S]$ by Theorem \ref{CPW16} (1). 
Hence, we assume $\rho (S)>1$ in what follows. 
Meanwhile, if $(d, {\rm Dyn} (S)) = (3,4A_1)$, then we obtain $\Ampc (S) = \Amp (S)$ by Lemma \ref{lem(5-1)}. 
From now on, we thus assume that $(d, {\rm Dyn} (S)) \not= (3,4A_1)$. 
Let $f:\wS \to S$ be the minimal resolution. 
By Proposition \ref{prop(3)}, there exists a non-negative integer $n$ with $5-d \le n \le 2$ such that $\wS$ satisfies $(\ast _n)$ (see the beginning of \S \ref{3}, for this definition). 
Notice that all irreducible curves with self-intersection number $\le -2$ are contracted by $f$. 
Therefore, we know $\Ampc (S) = \Amp (S)$ by Theorem \ref{thm(4)}. 
The proof of Theorem \ref{main(1)} is thus completed. 
\end{proof}


\begin{thebibliography}{99}
\bibitem{Bel23} G. Belousov, {\em Cylinders in del Pezzo surfaces of degree two}, In:\ {\em Birational Geometry, K\"{a}hler-Einstein Metrics and Degenerations}, Springer Proc. Math. Stat., Vol. 409, Springer, Cham, 2023, 17--70. 
\bibitem{Che21} I. Cheltsov, {\em Cylinders in rational surfaces}, Sb. Math. {\bf 212} (2021), 399--415. 
\bibitem{CDP18} I. Cheltsov, A. Dubouloz and J. Park, {\em Super-rigid affine Fano varieties}, Compos. Math. {\bf 154} (2018), 2462--2484. 
\bibitem{CPPZ21} I. Cheltsov, J. Park, Y. Prokhorov and M. Zaidenberg, {\em Cylinders in Fano varieties}, EMS Surv. Math. Sci. {\bf 8} (2021), 39--105. 
\bibitem{CPW16a} I. Cheltsov, J. Park and J. Won, {\em Affine cones over smooth cubic surfaces}, J. Eur. Math. Soc. {\bf 18} (2016), 1537--1564. 
\bibitem{CPW16b} I. Cheltsov, J. Park and J. Won, {\em Cylinders in singular del Pezzo surfaces}, Compos. Math. {\bf 152} (2016), 1198--1224. 
\bibitem{CPW17} I. Cheltsov, J. Park and J. Won, {\em Cylinders in del Pezzo surfaces}, Int. Math. Res. Not. {\bf 2017} (2017), 1179--1230. 
\bibitem{CT88} D. F. Coray and M. A. Tsfasman, {\em Arithmetic on singular Del Pezzo surfaces}, Proc. Lond. Math. Soc. (3) {\bf 57} (1988), 25--87. 
\bibitem{Dol12} I. V. Dolgachev, {\em Classical Algebraic Geometry: a modern view}, Cambridge Univ. Press, Cambridge, 2012. 
\bibitem{HHT22} N. T. A. Hang, M. Hoff, and H. L. Truong, {\em On cylindrical smooth rational Fano fourfolds}, J. Korean Math. Soc. {\bf 59} (2022), 87--103. 
\bibitem{HT23} N. T. A. Hang and H. L. Truong, {\em The affine cones over Fano-Mukai fourfold of genus $7$ are flexible}, Int. Math. Res. Not. {\bf 2024} (2024), 8417--8426. 
\bibitem{HT} M. Hoff, and H. L. Truong, {\em Flexibility of Affine cones over Mukai fourfolds of genus $g \ge 7$}, preprint, \href{https://arxiv.org/abs/2208.09109}{arXiv:2208.09109} (2022). 
\bibitem{KP21} J. Kim and J. Park, {\em Generic flexibility of affine cones over del Pezzo surfaces of degree $2$}, Int. J. Math. {\bf 32} (2021), Article ID 2150104, 18pp. 
\bibitem{KPZ11} T. Kishimoto, Y. Prokhorov and M. Zaidenberg, {\em Group actions on affine cones}, In: {\em Affine Algebraic Geometry}, CRM Proc. Lecture Notes, Vol. 54, Amer. Math. Soc., Providence, RI, 2011, 123--163. 
\bibitem{KPZ14} T. Kishimoto, Y. Prokhorov and M. Zaidenberg, {\em Unipotent group actions on del Pezzo cones}, Algebr. Geom. {\bf 1} (2014), 46--56. 
\bibitem{Koj02} H. Kojima, {\em Algebraic compactifications of some affine surfaces}, Algebra Colloq. {\bf 9} (2002), 417--425. 
\bibitem{MW18} L. Marquand and J. Won, {\em Cylinders in rational surfaces}, Eur. J. Math. {\bf 4} (2018), 1161--1196. 
\bibitem{Par22} J. Park, {\em $\mathbb{G}_a$-actions on the complements of hypersurfaces}, Transform. Groups {\bf 27} (2022), 651--657. 
\bibitem{PW16} J. Park and J. Won, {\em Flexible affine cones over del Pezzo surfaces of degree $4$}, Eur. J. Math. {\bf 2} (2016), 304--318. 
\bibitem{Pre13} A. Perepechko, {\em Flexibility of affine cones over del Pezzo surfaces of degree $4$ and $5$}, Funct. Anal. Appl. {\bf 47} (2013), 284--289. 
\bibitem{Pre21} A. Perepechko, {\em Affine cones over cubic surfaces are flexible in codimension one}, Forum math. {\bf 33} (2021), 339--348. 
\bibitem{Pre} A. Perepechko, {\em Generic flexibility of affine cones over del Pezzo surfaces in Sagemath}, preprint, \href{https://arxiv.org/abs/2305.06462}{arXiv:2305.06462} (2023). 
\bibitem{Won22} J. Won, {\em Flexibility of affine cones over singular del Pezzo surfaces with degree $4$}, East Asian Math. J. {\bf 38} (2022), 321--329. 
\bibitem{Zha88} D.-Q. Zhang, {\em Logarithmic del Pezzo surfaces of rank one with contractible boundaries}, Osaka J. Math. {\bf 25} (1988) 461--497. 
\end{thebibliography}
\end{document}